\documentclass[11pt,a4paper]{amsart}

\usepackage{a4wide}
\usepackage{graphicx}
\usepackage{latexsym}
\usepackage{epsfig}
\usepackage{amssymb}
\usepackage{amstext,amsmath}
\usepackage{amsgen}
\usepackage{amsxtra}
\usepackage{amsgen}
\usepackage{amsthm}
\usepackage{tcolorbox}
\usepackage{mathrsfs,dsfont,scalerel}

\usepackage{hyperref}
\usepackage{comment}

\usepackage{bm}

\usepackage{mathtools}

\newtheorem{thm}{Theorem}[]
\newtheorem{prop}{Proposition}[section]
\newtheorem{lemma}[prop]{Lemma}
\newtheorem{cor}[prop]{Corollary}

\theoremstyle{remark}
\newtheorem{rem}[prop]{Remark}

\theoremstyle{definition}

\newtheorem{definition}[prop]{Definition}

\numberwithin{equation}{section}

\DeclareMathOperator{\dv}{div}
\DeclareMathOperator{\esssup}{ess \, sup}

\newcommand{\pe}{\mathrm{Pe}}

\renewcommand{\i}{\mathrm{i}}
\newcommand{\n}{\mathbf{n}}

\newcommand{\bxi}{\boldsymbol{\xi}}

\newcommand{\p}{{\bf p}}

\newcommand{\e}{{\bf e}}

\renewcommand{\d}{\, \mathrm{d} }
\newcommand{\de}{\mathrm{d}}

\newcommand{\per}{{\mathrm{per}}}

\newcommand{\id}{\mathrm{id}}
\renewcommand{\i}{\mathrm{i}}

\usepackage{xcolor}

\begin{document}
\title[Crowded active Brownian system with size-exclusion]{Well-posedness and stationary states for a crowded active Brownian system with size-exclusion}


\author[M.~Burger]{Martin Burger}
\address[M.~Burger]{Computational Imaging Group and Helmholtz Imaging, Deutsches Elektronen-Synchrotron (DESY), Notkestr. 85, 22607 Hamburg, Germany, and Fachbereich Mathematik, Universität Hamburg, Bundesstra{\ss}e 55, 20146 Hamburg, Germany}
\email{martin.burger@desy.de}

\author[S.~M.~Schulz]{Simon Schulz}
\address[S.~M.~Schulz]{Scuola Normale Superiore, Centro di Ricerca Matematica Ennio De Giorgi, P.zza dei Cavalieri, 3,  56126 Pisa, Italy}\email{simon.schulz@sns.it}

\maketitle
\begin{abstract}

%
We prove the existence of solutions to a non-linear, non-local, degenerate equation which was previously derived as the formal hydrodynamic limit of an active Brownian particle system, where the particles are endowed with a position and an orientation. This equation incorporates diffusion in both the spatial and angular coordinates, as well as a non-linear non-local drift term, which depends on the angle-independent density. The spatial diffusion is non-linear degenerate and also comprises diffusion of the angle-independent density, which one may interpret as cross-diffusion with infinitely many species. Our proof relies on
interpreting the equation as the perturbation of a gradient flow in a Wasserstein-type space. It generalizes the boundedness-by-entropy method to this setting  and makes use of a gain of integrability due to the angular diffusion. For this latter step, we adapt a classical interpolation lemma for function spaces depending on time. We also prove uniqueness in the particular case where the non-local drift term is null, and provide existence and uniqueness results for stationary equilibrium solutions. 

\vspace{0.2cm}


%
\end{abstract}


\setcounter{tocdepth}{1}
\tableofcontents

\section{Introduction}\label{sec:intro}

We study the following non-local non-linear degenerate equation, which was derived in \cite{bbesModel} as the formal hydrodynamic limit of a many-particle system modeling the interaction of active Brownian particles endowed with a two-dimensional position $x \in \Omega = (0,2\pi)^2$ and orientation $\theta \in (0,2\pi)$: 
\begin{equation}\label{eq: model 4}
    \begin{aligned}
    & \partial_t f + \pe \dv \big(f(1-\rho) \e(\theta)\big) = D_e\dv \big((1-\rho)\nabla f + f \nabla \rho\big) + \partial^2_\theta f,
    \end{aligned}
\end{equation}
where $\rho(t,x) = \int_0^{2\pi} f(t,x,\theta) \d \theta$ is the \emph{angle-independent density}, 
$\e(\theta)=(\cos\theta,\sin\theta)$, the real constant parameters $\pe$ and $D_e>0$ are respectively the P\'eclet number and spatial diffusion coefficient, supplemented with initial data $f_0$ and periodic boundary conditions in both the angle and space variables; the periodic cell for the space-angle coordinates is denoted by $\Upsilon := \Omega \times (0,2\pi)$. We prove the existence of weak solutions to \eqref{eq: model 4}, show uniqueness of such solutions in the case of zero P\'eclet number, and provide existence and uniqueness results for stationary states under various assumptions on the P\'eclet number.

Active Brownian systems model the motion of self-propelled particles and have an abundance of applications; for instance in biology, pedestrian and traffic flows, and the modeling of collective behaviour in general (\textit{cf.~e.g.~}\cite{Romanczuk:2012iz,Cates:2013ia,Redner.2013,Cates:2014tr,Stenhammar:2015ex,Yeomans:2015dt,Speck:2015um}). Much like kinetic models, active Brownian systems depict the \emph{mesoscopic} dynamics of a group of interacting agents and, indeed, equation \eqref{eq: model 4} presents similarities with kinetic Fokker--Planck equations as well as non-linear degenerate parabolic systems. At first glance, \eqref{eq: model 4} appears similar to the cross-diffusion system studied in \cite{BoundEntropy}: 
\begin{equation}\label{eq:martin marco cross diff}
    \begin{aligned}
    &\partial_t r = \dv\big( (1-\sigma)\nabla r + r \nabla \sigma + r(1-\sigma)\nabla V \big), \\ 
    &\partial_t b = \dv\big( (1-\sigma)\nabla b + b \nabla \sigma + b(1-\sigma)\nabla W \big), 
    \end{aligned}
\end{equation}
where $\sigma = r + b$ and $V,W$ are given external potentials. Equation \eqref{eq: model 4} may be interpreted as the analog of this system with infinitely-many species; the diffusive term $\dv f \nabla \rho$ encodes the diffusion of the uncountable family of angles in the interval $(0,2\pi)$. However, unlike the system in \cite{BoundEntropy}, the equation \eqref{eq: model 4} is non-local; both in its diffusive and drift terms. This non-locality is a feature shared by certain kinetic models, in particular the Vlasov--Benney equation \cite{vlasovBenney1,vlasovBenney2} and the recent contribution \cite{BriantMeunier}. Moreover, the angle-independent density, which formally satisfies the strongly parabolic drift-diffusion equation 
\begin{equation*}
    \partial_t \rho + \dv\big( (1-\rho) \p \big) = \Delta \rho, 
\end{equation*}
where $\p = \int_0^{2\pi} f \e(\theta) \d \theta$, is expected to be more regular than $f$ as its governing equation is no longer degenerate. This mechanism is evocative of velocity averaging lemmas (see, \textit{e.g.}, \cite{PerthameKinetic}), illustrating another parallel between kinetic theory and active Brownian systems. The zones of degeneracy $\{\rho=1\}$ and $ \{f = 0\}$ comprised in \eqref{eq: model 4} make its analysis delicate as one expects the formation of discontinuities and other losses of regularity along these interfaces; this first degenerate region is all-the-more troublesome because one can only infer information about the local behaviour of $\rho$, not $f$. In turn, there exists relatively few known results in the existing literature for such equations.

To the authors' knowledge, there exists only one prior contribution studying equations similar to \eqref{eq: model 4} (except the authors' own work \cite{bbes}, which has a different diffusion model and thus less degeneracy), namely the result of Erignoux \cite[Theorem 2.3.3]{Erignoux:2016un} . Therein, the author relies solely on probabilistic techniques to  compute the hydrodynamic limit of the underlying particle system, thereby establishing the existence of a very weak solution; defined only as a measure. The present contribution provides an entirely deterministic PDE-theoretic proof, and we show the existence of a Sobolev weak solution with stronger regularity properties by means of the boundedness-by-entropy principle and a vanishing diffusivity method. Our strategy relies on using the linear diffusion in angle to upgrade the integrability of $f$ and obtain improved compactness properties; for this we adapt a classical interpolation lemma (\textit{cf.}~Appendix \ref{sec:appendix interpolation}) 
. We also prove uniqueness of our solution within the stronger class of admissible weak solutions when the P\'eclet number is assumed to be null, and provide existence and uniqueness results for stationary solutions.

At the core of the aforementioned boundedness-by-entropy method (originally developed in \cite{BoundEntropy}) is the notion that equation \eqref{eq: model 4} may be interpreted as a perturbation of a Wasserstein gradient flow associated to the entropy functional 
\begin{equation*}\label{eq:entropy functional}
    E[f] = \int_\Upsilon \big( f \log f + (1-\rho)\log(1-\rho) \big) \d \bxi; 
\end{equation*}
see \S \ref{sec:calc var} for further details. One is able to obtain additional \emph{a priori} estimates on the equation by considering the natural dual problem from the point of view of the underlying variational evolution. That is to say, we consider the first variation $u = E'[f]$ and rewrite \eqref{eq: model 4} as an equation on $u$, which reads as a perturbation of a gradient flow associated to the convex conjugate functional $E^*[u]$. By rewriting the original unknown $f$ in terms of $u$, one is able to immediately determine the non-negativity of $f$ and the \emph{a priori} bounds $0 \leq \rho \leq 1$. This generalizes the boundedness-by-entropy method developed for the related cross-diffusion system \eqref{eq:martin marco cross diff} in \cite{BoundEntropy} and later generalized in \cite{jungel2015boundedness}. The underlying philosophy is analogous to the reformulation of Lagrangian to Hamiltonian mechanics; now written in the context of Wasserstein gradient flows with nontrivial mobility. 

The rest of the paper is organised as follows. In \S \ref{sec:main results} we state our main results precisely, and outline our strategy in \S \ref{sec:setup}. \S \ref{sec:calc var} introduces the boundedness-by-entropy method; which we use to construct a sequence of suitable transformed regularised problems related to \eqref{eq: model 4}. In \S \ref{sec:galerkin}, we prove the existence of the Galerkin approximation to the aforementioned transformed regularised problems and obtain uniform estimates. In \S \ref{sec:proof of existence result}, we give the proof of our main existence result. In \S \ref{sec:uniqueness zero pe}, we prove uniqueness of the solution obtained in \S \ref{sec:proof of existence result} under the assumption that the P\'eclet number is null. Finally, in \S \ref{sec:stat states}, we prove the existence and uniqueness of stationary states under particular assumptions on the P\'eclet number. Appendix \ref{sec:appendix interpolation} contains an interpolation lemma pertaining to the uniform estimates of \S \ref{sec:unif est}, while Appendix \ref{sec:appendix initial data} contains technical results related to the sequence of regularised initial data used for the existence proof. 

\vspace{0.2cm}

\textbf{Notation: } Throughout this work, the letter $C$ will denote a positive constant independent of the parameters $\varepsilon,n$, which may change from line to line. Unless stated explicitly otherwise, the symbols $\dv,\nabla,\Delta$ denote divergences, gradients, and Laplacians taken with respect to the space variable $x$, while the symbols $\nabla_{\bxi},\Delta_{\bxi}$ denote gradients and Laplacians with respect to the concatenated space-angle variable $\bxi=(x,\theta)$. The domain on which the equation is posed will be denoted by $\Upsilon_T := (0,T)\times\Upsilon$, the closure of which we denote by $\overline{\Upsilon}_T := [0,T]\times\overline{\Upsilon}$; the sets $\Omega_T$ and $\overline{\Omega}_T$ are defined analogously. In later estimates and weak formulations, we shall make use of the families of admissible test functions $\mathcal{A},\mathcal{A}_s$, defined by 
\begin{equation}\label{eq:admissible test functions}
   \begin{aligned}
       &\mathcal{A} := \big\{ \psi \in C^\infty([0,T]\times\mathbb{R}^3) :  \, \psi(t,\cdot) \text{ triply } 2\pi \text{-periodic in }x,\theta, \, \psi(0,\cdot)=\psi(T,\cdot)=0 \big\}, \\ 
       &\mathcal{A}_s := \big\{ \psi \in C^\infty([0,T]\times\mathbb{R}^2) :  \, \psi(t,\cdot) \text{ doubly } 2\pi\text{-periodic in }x, \, \psi(0,\cdot)=\psi(T,\cdot)=0 \big\}, 
   \end{aligned} 
\end{equation}
and test functions will always be denoted by the letter $\psi$. We define the function spaces 
\begin{equation*}
    \begin{aligned}
    &X := L^6(0,T;W^{1,6}_\per(\Upsilon)), \quad Y := L^2(0,T;H^1_\per(\Omega)), \quad Z:= W^{1,6}_\per(\Upsilon), \\ 
    &\begin{aligned}\mathcal{X} := \Big\{ 
        f \in & L^3(0,T;L^3_\per(\Upsilon)): f \geq 0 \text{ a.e.}, \partial_\theta \sqrt{f} \in L^2(\Upsilon_T), \partial_t f \in X',  \\ 
        &\rho = {\int_0^{2\pi}} f \d \theta \in [0,1] \text{ a.e.}, \partial_t \rho \in Y',
        \nabla\sqrt{1-\rho} \in L^2(\Omega_T) , \sqrt{1-\rho}\nabla \sqrt{f} \in L^{2}(\Upsilon_T)\Big\}, 
    \end{aligned}
    \end{aligned}
\end{equation*}
and, in what follows, the notation $\langle \cdot , \cdot \rangle$ will always refer to the obvious underlying duality product where there is no confusion, unless stated explicitly otherwise. In the previous definitions and in what follows, the functions spaces denoted by $Z_\per(S)$ for $S \in \{(0,2\pi)^d\}_{d=1}^3$ are understood to mean 
\begin{equation*}
    Z_\per(S) := \big\{ g:\mathbb{R}^d \to \mathbb{R}: \, \Vert g \Vert_{Z(S)} < \infty, \text{ and } g(y+2\pi \e_i) = g(y) \, \forall y \in \mathbb{R}^d, i \in \{1,\dots,d\}  \big\}, 
\end{equation*}
where $\{\e_i\}_{i=1}^d$ is the standard basis of $\mathbb{R}^d$. The case $d=2$ corresponds to functions depending on $x$ and $d=3$ to functions depending on $x,\theta$. 

\subsection{Main theorems}\label{sec:main results}

In order to state our main existence result, we begin by introducing our notion of solution.

\begin{definition}[Weak solution]\label{def:weak sol}
   We say that $f \in \mathcal{X}$ is a \emph{weak solution} of \eqref{eq: model 4} with periodic boundary conditions if, for all $\psi \in X$, there holds 
  \begin{equation}\label{eq:weak sol}
    \begin{aligned}
      -\langle \partial_t f, \psi \rangle + \pe \int_{\Upsilon_T} (1-\rho)& f \e(\theta) \cdot \nabla \psi \d \bxi \d t \\ 
      &=\int_{\Upsilon_T} \Big( D_e\big((1-\rho)\nabla f + f \nabla \rho \big) \cdot \nabla \psi + \partial_\theta f\partial_\theta \psi \Big) \d \bxi \d t. 
    \end{aligned}
\end{equation}
\end{definition}

    We note that all integrals in the above weak formulation are well-defined since, using the Jensen and H\"older inequalities, 
    \begin{equation}\label{eq:weak form is well defined in intro}
        \begin{aligned}
            &\partial_\theta f \partial_\theta \psi = 2 \underbrace{\sqrt{f}}_{\in L^6} \underbrace{\partial_\theta \sqrt{f}}_{\in L^2} \underbrace{\partial_\theta \psi}_{\in L^6} \in L^{\frac{6}{5}}(\Upsilon_T) \subset L^1(\Upsilon_T) \\ 
            & f \nabla \rho \cdot \nabla \psi = -2 \underbrace{f}_{\in L^3}\underbrace{\sqrt{1-\rho}}_{\in L^\infty}\underbrace{\nabla\sqrt{1-\rho}}_{L^2} \cdot \underbrace{\nabla \psi}_{L^6} \in L^1(\Upsilon_T) \\ 
            & (1-\rho)\nabla f \cdot \nabla \psi = 2\underbrace{\sqrt{1-\rho}}_{\in L^\infty} \underbrace{\sqrt{f}}_{\in L^6} \underbrace{\sqrt{1-\rho}\nabla \sqrt{f}}_{\in L^2} \cdot \underbrace{\nabla \psi}_{\in L^6} \in L^{\frac{6}{5}}(\Upsilon_T) \subset L^1(\Upsilon_T). 
        \end{aligned}
    \end{equation}

Note that one may identify the duality product with 
\begin{equation*}
    \langle \partial_t f , \psi \rangle = \int_0^T \langle \partial_t f(t,\cdot), \psi(t,\cdot) \rangle_{Z'\times Z} \d t, 
\end{equation*}
whence a standard argument (\textit{cf.}~\cite{AGS08}) shows that the weak formulation of Definition \ref{def:weak sol} is equivalent to requiring that, for all $\psi \in \mathcal{A}$, there holds 
 \begin{equation}\label{eq:alt dt on test function}
    \begin{aligned}
     \int_{\Upsilon_T} f \partial_t \psi \d \bxi \d t + \pe \int_{\Upsilon_T} (1-\rho)& f \e(\theta) \cdot \nabla \psi \d \bxi \d t \\ 
      &=\int_{\Upsilon_T} \Big( D_e\big((1-\rho)\nabla f + f \nabla \rho \big) \cdot \nabla \psi + \partial_\theta f\partial_\theta \psi \Big) \d \bxi \d t. 
    \end{aligned}
\end{equation}

Our main existence result is the following. 
    
\begin{thm}[Existence of weak solution]\label{thm:main existence result}
     Let $T>0$ be arbitrary. Let $f_0$ be measurable non-negative initial data such that $f_0 \in L^p(\Upsilon)$ for some $p>1$ and $\rho_0 = \int_0^{2\pi} f_0 \d \theta \in [0,1]$ a.e.~in $\Omega$. Then, there exists $f \in \mathcal{X}$ a weak solution of \eqref{eq: model 4} in the sense of Definition \ref{def:weak sol}. Furthermore, the initial data is attained in the sense $\lim_{t \to 0^+}\Vert f(t,\cdot) - f_0 \Vert_{Z'} = 0$. 
\end{thm}

The essential requirement on $f_0$ in the above theorem is the finiteness of the entropy functional $E[f_0]$. This is of course satisfied by imposing the stronger condition $f_0 \in L^p(\Upsilon)$ for some $p>1$. For clarity of exposition, we choose this latter requirement in order to frame the problem in Lebesgue spaces instead of Orlicz spaces. We note in passing that the weak solution obtained in Theorem \ref{thm:main existence result} also satisfies an alternative notion of weak solution which is more convenient when studying regularity (\textit{cf.}~\cite{regularity}); this alternative weak formulation is written at the end of \S \ref{sec:proof of existence result}. 

\begin{rem}[Attainment of initial data]\label{rem:ibp for duality bracket}
Note that $\mathcal{X} \subset W^{1,\frac{6}{5}}(0,T;Z')$, whence (see, \textit{e.g.}, \cite[Theorem 2, \S 5.9.2]{evans}) $f \in C([0,T];Z')$ and $\Vert f(t,\cdot) - f(s,\cdot) \Vert_{Z'} \leq |t-s|^{1/6}C( \Vert \partial_t f \Vert_{Y'})$ for a.e.~$0 < s < t < T$. Hence, the attainment of the initial data is phrased in the sense of $Z'$. 
\end{rem}

Our main uniqueness result is the following, and focuses on the case of zero P\'eclet number. 

\begin{thm}[Uniqueness for null P\'eclet number]\label{thm:uniqueness}
    If $\pe=0$, then the weak solution provided by Theorem \ref{thm:main existence result} is unique in $\mathcal{X}$. 
\end{thm}

Finally, we study the existence and uniqueness of stationary states of \eqref{eq: model 4} in \S \ref{sec:stat states}. We mostly concentrate on uniqueness, which we divide into two parts: uniqueness for zero P\'eclet number, and uniqueness of strong solutions near \emph{constant} stationary states for general P\'eclet number. We delay the statements of these results to \S \ref{sec:stat states}.

\subsection{Strategy and set-up}\label{sec:setup}

In this section, we describe our strategy of proof for the main existence result Theorem \ref{thm:main existence result} in \S \ref{sec:strategy}, and then describe our mollification procedure for regularising the initial data in \S \ref{sec:setup initial}. 

\subsubsection{Strategy}\label{sec:strategy}

We briefly outline our strategy of proof. As already mentioned, the equation \eqref{eq: model 4} may be interpreted as a perturbation of a Wasserstein gradient flow associated to the entropy functional $E$. It is therefore natural to reframe the problem in terms of the first variation of $E[f]$, which we denote by $u = E'[f]$; see \S \ref{sec:calc var} for more details. The resulting equation for $u$ is again non-linear, non-local, and degenerate-parabolic. We therefore add the regularisation term $\varepsilon \Delta_{\bxi} u$ to this equation, and compute Galerkin approximations of the solution of this equation, which we denote by $u^\varepsilon_n$; where $n$ denotes the $n$-th Galerkin approximation. These approximations give rise to approximations $f^\varepsilon_n$ for the solution of \eqref{eq: model 4} by a simple explicit formula. After proving the required uniform estimates for the families of functions $\{u^\varepsilon_n,f^\varepsilon_n,\rho^\varepsilon_n\}_{n,\varepsilon}$, which are independent of $n$ and $\varepsilon$, we perform a diagonal procedure so as to simultaneously take the limit $n\to\infty$ and $\varepsilon\to 0$, and we obtain a weak solution of \eqref{eq: model 4} in the sense of Definition \ref{def:weak sol}.

\subsubsection{Set-up for initial data}\label{sec:setup initial}

We shall approximate the initial data with smooth strictly positive functions. We consider the product mollifier sequence 
\begin{equation*}
    \eta_\varepsilon(x,\theta) := \alpha_\varepsilon(x)\beta_\varepsilon(\theta), 
\end{equation*}
where $\alpha,\beta$ are compactly supported smooth non-negative bump functions with unit integral; $\alpha_\varepsilon(x) := \varepsilon^{-2} \alpha(x/\varepsilon)$, $\beta_\varepsilon(\theta) = \varepsilon^{-\gamma} \beta(\theta/\varepsilon^\gamma)$ for $\gamma>2$. Define, for $\varepsilon \in (0,1)$, the strictly positive periodic functions 
\begin{equation*}
   \begin{aligned}
       &f_0^\varepsilon := \frac{\varepsilon}{4\pi} + (1-\varepsilon)\eta_\varepsilon*f_0 \in C^\infty_\per(\overline{\Upsilon}), \\ 
       &\rho_0^\varepsilon := \int_0^{2\pi} f_0^\varepsilon(\cdot,\theta) \d \theta = \frac{\varepsilon}{2} + (1-\varepsilon)\alpha_\varepsilon*\rho_0 \in C^\infty_\per(\overline{\Omega}), 
   \end{aligned} 
\end{equation*}
where the convolutions are understood in the sense $\eta_\varepsilon*f_0 = \int_{\mathbb{R}^3} \eta_\varepsilon(\cdot-\bxi)  f_0(\bxi) \d \bxi$ and $\alpha_\varepsilon*\rho_0 = \int_{\mathbb{R}^2} \alpha_\varepsilon(\cdot-x)\rho_0(x) \d x$. A straightforward computation (\textit{cf.}~Appendix \ref{sec:appendix initial data}) shows that $0<f^\varepsilon_0$ and $0<\rho^\varepsilon_0 < 1$, whence the function 
\begin{equation*}
    \begin{aligned}
  u_0^\varepsilon := \log\Big( \frac{f^\varepsilon_0}{1-\rho^\varepsilon_0} \Big) \in C^\infty_\per(\overline{\Upsilon}) 
    \end{aligned}
\end{equation*}
is also well-defined. The next lemma summarises the key properties of the regularised initial data. The proof is delayed to Appendix \ref{sec:appendix initial data}. 

\begin{lemma}[Properties of the regularised initial data]\label{lem:regularised initial data}
    There exists a subsequence of the regularised initial data sequence $\{f^\varepsilon_0,\rho^\varepsilon_0,u^\varepsilon_0\}_\varepsilon$, which we do not relabel, such that 
    \begin{equation*}
    f^\varepsilon_0 \to f_0 \text{ strongly in } L^p(\Upsilon), \quad \rho^\varepsilon_0 \to \rho_0 \text{ strongly in } L^q(\Omega) \text{ for all } q \in [1,\infty), 
    \end{equation*}
    and for which there exists a positive constant $C=C(p,\Upsilon)$, independent of $\varepsilon$, such that 
    \begin{equation*}
        \begin{aligned}
            E[f^\varepsilon_0] + \varepsilon \int_\Upsilon |u^\varepsilon_0|^2 \d \bxi \leq C\big(1+\Vert f_0 \Vert_{L^p(\Upsilon)} \big). 
        \end{aligned}
    \end{equation*}
\end{lemma}

\section{Entropy Variable and Transformed Problem}\label{sec:calc var}

In the following we will perform a transformation of the problem \eqref{eq: model 4} into entropy variables, which diagonalizes the diffusion part at the cost of coupling the time derivatives of $f$ and $\rho$. This transformation allows for a suitable approximation via the Galerkin method, as it did for the cross-diffusion system \eqref{eq:martin marco cross diff} considered in \cite{BoundEntropy}, and an application of the boundedness-by-entropy principle; developed again in \cite{BoundEntropy} and subsequently generalised in \cite{jungel2015boundedness}.

For this sake we consider the aforementioned \emph{entropy functional} of \eqref{eq:entropy functional}, \textit{i.e.}, 
\begin{equation*}
    E[f] = \int_\Upsilon \big(  f \log f + (1-\rho) \log(1-\rho) \big) \d x \d \theta, 
\end{equation*}
where $\rho = \int_0^{2\pi} f \d \theta$. Note that  $E[f]$ is bounded from below by a fixed constant for all $f$, provided $0 \leq \rho \leq 1$ and $f \geq 0$ a.e. We can define $E$ on a dense subspace of $L^p(\Upsilon)$ for any $p > 1$, \textit{e.g.}, smooth functions, and extend it uniquely to the whole space. Note that $E$ is a convex functional, and we may define its convex conjugate, or equivalently its Legendre transform, $E^*: L^{p'}(\Upsilon) \to \mathbb{R}$
via  
$$ E^*[u] := \sup\big\{ \langle u, f \rangle - E[f]: f \in L^p(\Upsilon)\big\},$$
which is again a convex functional. In fact, we may compute $E^*$ explicitly as 
$$ E^*[u] = \int_\Omega \log \left( 1 + \int_0^{2\pi}  e^u \d\theta\right) \d x. $$ As usual we define the \emph{entropy variable} (or dual variable) related to $f$ via the functional derivative 
\begin{equation}\label{eq:entropy variable def}
u := E'[f] = \log f - \log (1-\rho). 
\end{equation}
We therefore write $f$ and $\rho$ as functions of $u:\Upsilon_T \to \mathbb{R}$ via the expressions: 
\begin{equation*}
    \begin{aligned}
       f = (1-\rho)e^{u}, \qquad        \rho = (1-\rho) \int_0^{2\pi} e^{u} \d \theta. 
    \end{aligned} 
\end{equation*}
It follows that 
\begin{equation}\label{eq:rho in terms of entropy variable}
    \rho(t,x) = \frac{\int_0^{2\pi} e^{u(t,x,\theta)} \d \theta}{1+\int_0^{2\pi} e^{u(t,x,\theta)} \d \theta} 
\end{equation}
automatically satisfies the bound $0 \leq \rho \leq 1$. In turn, we obtain 
\begin{equation}
    f(t,x,\theta) = \frac{e^{u(t,x,\theta)}}{1+\int_0^{2\pi} e^{u(t,x,\theta)} \d \theta}. 
\end{equation}
Note that this latter formula makes the relation $(E^*)' = (E')^{-1}$ explicit, \textit{i.e.},
$$ u = E'[f] , \quad f = (E^*)'[u].$$

The main idea to analyze the system is as follows: the original problem \eqref{eq: model 4} is considered as a perturbation of the gradient flow 
$$  \partial_t f =  \nabla_{\bxi} \cdot \left( M(f,\rho)  (\nabla_{\bxi}  E'[f] +  V) \right) $$
with mobility $M$ and perturbative vector field $V$, independent of $f$, given by 
$$M (f,\rho) = \left( \begin{array}{cc} D_e f (1-\rho) I  & 0 \\ 0 & f\end{array}\right), \qquad V(x,\theta) = \left(\begin{array}{c}  - \pe~\e(\theta) \\0 \end{array} \right).$$
Transforming to the entropy variable we arrive at 
\begin{equation}\label{eq:grad flow perturbation in terms of u}  
\partial_t ((E^*)'[u]) =  \nabla_{\bxi} \cdot \left( \tilde M(u)  (\nabla_{\bxi} u +  V) \right), 
\end{equation}
where $\tilde M(u) = M(f[u],\rho[u])$ is given by 
\begin{equation}\label{eq:tilde M def}
    \tilde{M}(u) = \left( \begin{matrix}
    \frac{e^u}{(1+\int_0^{2\pi} e^u \d\theta)^2} D_e I & 0 \\ 
      0 & \frac{e^u}{1+\int_0^{2\pi} e^u \d \theta}
    \end{matrix}\right). 
\end{equation}
This dual formulation has several advantages. First of all, the boundedness-by-entropy principle applies, \textit{i.e.}~if we find a weak solution $u$ of the above, it is apparent that after transformation $f \geq 0$ and $0 \leq \rho \leq 1$ hold almost everywhere.
Moreover, the transformed system can be approximated more easily, \textit{e.g.}~by regularising with a multiple of the heat operator. This yields a problem of the form
\begin{equation}\label{eq:eps eqn}
(\varepsilon \id + A(u^\varepsilon)) [\partial_t u^\varepsilon] =  \nabla_{\bxi} \cdot \left( (\varepsilon I + \tilde M(u^\varepsilon))  \nabla_{\bxi} u^\varepsilon +  \tilde M(u^\varepsilon) V \right), 
\end{equation}
where, for all $u \in L^{p'}(\Upsilon)$, the quantity $A(u)$ is a linear operator which corresponds to the second functional derivative of $E^*$ evaluated on the element $u$; furthermore, since $E^*$ is convex, this quantity is a positive semi-definite linear operator. In fact, one can explicitly write 
\begin{equation}\label{eq:explicit Hessian}
    \begin{aligned}
        A(u^\varepsilon)[\partial_t u^\varepsilon] = \frac{e^{u^\varepsilon}}{1+\int_0^{2\pi} e^{u^\varepsilon} \d \theta} \partial_t u^\varepsilon - \frac{e^{u^\varepsilon}}{\big( 1+\int_0^{2\pi} e^{u^\varepsilon} \d \theta \big)^2}\int_0^{2\pi} e^{u^{\varepsilon}} \partial_t u^\varepsilon \d \theta. 
    \end{aligned}
\end{equation}

In principle, one can now write a proof of existence and uniqueness including some regularity for the regularised problem \eqref{eq:explicit Hessian}; this is the content of \S \ref{sec:galerkin}. Moreover, testing the equation with $u$ preserves \emph{a priori} estimates for the entropy, noticing that 
\begin{equation}\label{eq:switching from duality bracket to entropy dissipation}
\frac{\de}{\de t} E[f] = \langle E'(f), \partial_t f \rangle_\Upsilon = \langle u, \partial_t ((E^*)'[u]) \rangle_{\Upsilon}. 
\end{equation}

\section{Galerkin Approximation for the Transformed Regularised System}\label{sec:galerkin}

The main result of this section is as follows. We shall prove it in several steps: first, we obtain the existence of the Galerkin approximations to the regularised problem \eqref{eq:eps eqn} in \S \ref{sec:existence galerkin}, and then we obtain the desired uniform estimates in \S \ref{sec:unif est}. 

\begin{prop}[Existence and uniform estimates for Galerkin approximations]\label{prop:transformed reg sys existence and est}
   Let $\varepsilon \in (0,1)$ and $n \in \mathbb{N}$. There exists functions $u^\varepsilon_n,f^\varepsilon_n,\rho^\varepsilon_n \in C^1(\overline{\Upsilon}_T)$ satisfying the relations 
\begin{equation}\label{eq:ae rel for reg sys}
    f^\varepsilon_n(t,x,\theta) = \frac{e^{u^\varepsilon_n(t,x,\theta)}}{1+\int_0^{2\pi} e^{u^\varepsilon_n(t,x,\theta)} \d \theta}, \qquad \rho^\varepsilon_n(t,x) = \int_0^{2\pi} f^\varepsilon_n(t,x,\theta) \d \theta \in [0,1] 
\end{equation}
   in $\Upsilon_T$, such that, for all $\psi \in \mathcal{A}$, there holds 
\begin{equation}\label{eq:weak form reg sys}
    \begin{aligned}
  - \langle \partial_t (\varepsilon u^\varepsilon_n + &f^\varepsilon_n) , \psi \rangle + \pe\int_{\Upsilon_T} (1-\rho^\varepsilon_n)f^\varepsilon_n \e(\theta) \cdot \nabla \psi \d \bxi \d t \\ 
      =&\int_{\Upsilon_T} \Big( \varepsilon \nabla_{\bxi}u^\varepsilon_n \cdot \nabla_{\bxi}\psi + D_e\big((1-\rho^\varepsilon_n)\nabla f^\varepsilon_n + f^\varepsilon_n \nabla \rho^\varepsilon_n \big) \cdot \nabla \psi + \partial_\theta f^\varepsilon_n \partial_\theta \psi \Big) \d \bxi \d t, 
    \end{aligned}
\end{equation}
Furthermore there exists a positive constant $C$, independent of $\varepsilon$ and $n$, such that 
 \begin{equation}\label{eq:main est reg i}
    \begin{aligned}
      \esssup_{t \in [0,T]} E[f^\varepsilon_n](t) +  \esssup_{t \in [0,T]}\int_\Upsilon |\sqrt{\varepsilon} u^\varepsilon_n(t)|^2 \d \bxi + \int_{\Upsilon_T} |\nabla_{\bxi} \sqrt{\varepsilon} u^\varepsilon_n|^2 \d \bxi \d t& \\ 
      + \int_{\Upsilon_T} (1-\rho^\varepsilon_n)|\nabla_{\bxi} \sqrt{f^\varepsilon_n} |^2  \d \bxi \d t + \int_{\Omega_T} |\nabla \sqrt{1-\rho^\varepsilon_n}|^2 \d x \d t& \\ 
      + \int_{\Upsilon_T} |\partial_\theta  \sqrt{f^\varepsilon_n}|^2 \d \bxi \d t + \int_{\Omega_T} |\nabla \rho^\varepsilon_n|^2  \d x \d t & \leq C, 
        \end{aligned}
\end{equation}
as well as 
\begin{align}
     \Vert f^\varepsilon_n \Vert_{L^3(\Upsilon_T)} \leq C \label{eq:sqrt func di ben reg} \\ 
     \Vert \sqrt{(1-\rho^\varepsilon_n)f^\varepsilon_n} \Vert_{L^{\frac{10}{3}}(\Upsilon_T)} \leq C \label{eq:gradient of f sqrt 1- rh rego} \\ 
     \Vert f^\varepsilon_n \sqrt{1-\rho^\varepsilon_n} \Vert_{L^{\frac{6}{5}}(0,T;W^{1,\frac{6}{5}}(\Upsilon))} \leq C, \label{eq:bound on f sqrt 1-rho main one for limit reg} 
     \end{align}
     and the time-derivative bounds 
     \begin{align}
     \Vert \partial_t(\varepsilon u^\varepsilon_n + f^\varepsilon_n) \Vert_{X'} \leq C, \label{eq:main time deriv reg i} \\ 
        \Vert \partial_t(\varepsilon U^\varepsilon_n + \rho^\varepsilon ) \Vert_{Y'} \leq C, \label{eq:main time deriv reg ii}
\end{align}
where $U^\varepsilon_n(t,x) := \int_0^{2\pi} u^\varepsilon_n(t,x,\theta) \d \theta$ also satisfies 
\begin{equation}\label{eq:Ueps unif bound}
   \Vert \sqrt{\varepsilon} U^\varepsilon_n \Vert_{L^\infty(0,T;L^2(\Omega))} + \Vert \sqrt{\varepsilon} U^\varepsilon_n \Vert_{L^2(0,T;H^1(\Omega))} \leq C. 
\end{equation} 
\end{prop}

\subsection{Existence of Galerkin approximants}\label{sec:existence galerkin}

In what follows, we use that the biharmonic operator $\Delta_{\bxi}^2$, considered with periodic boundary conditions in space and angle on $\Upsilon$, is a coercive, bounded, linear self-adjoint operator with respect to the inner product of $H^2_\per(\Upsilon)$, where 
\begin{equation*}
    \langle f, g \rangle_{H^2_\per(\Upsilon)} = \int_\Upsilon \Delta_{\bxi}f \Delta_{\bxi} g \d \bxi. 
\end{equation*}
It follows from the classical spectral theorem on Hilbert spaces that $\Delta_{\bxi}^2$ admits a discrete spectrum $0 < \lambda_1 < \lambda_2 < \dots$~with eigenfunctions $\{\varphi_j\}_{j=1}^\infty$ such that there holds $\Delta_{\bxi}^2 \varphi_j = \lambda_j \varphi_j$ for all $j$; a classical bootstrapping procedure ensures that $\varphi_j \in C^\infty_\per(\overline{\Upsilon})$ for all $j$, \textit{cf.}~\textit{e.g.}~\cite[Chapitre 7]{AllaireCours}. Furthermore, $\{\varphi_j\}_{j=1}^\infty$ is an orthonormal basis of $H^2_\per(\Upsilon)$. 

We note in passing that our choice of functional setting $H^2_\per(\Upsilon)$ is motivated by the fact that we must control the $L^\infty$ norm of the sequence of Galerkin approximations $\{u_n^\varepsilon(0,\cdot)\}_n$ in order to estimate the approximations $\{f_n^\varepsilon(0,\cdot)\}_n$ via the entropy variable formula \eqref{eq:ae rel for reg sys}; no $L^p$ control for $p < \infty$ suffices. This is made possible by approximating $u^\varepsilon(0,\cdot)$ in $H^2_\per(\Upsilon)$ and using Morrey's inequality over the three-dimensional open domain $\Upsilon$; see Step 1 of the proof of Lemma \ref{lem:unif est for galerkin} for details.

For each fixed $n\in\mathbb{N}$ we define the closed subspace $\mathcal{V}_n := \overline{\mathrm{span}}\{\varphi_j\}_{j=1}^n$. Note from our previous discussion that $\bigcup_n \mathcal{V}_n$ is dense in $H^2_\per(\Upsilon)$, $\mathcal{V}_n \subset \mathcal{V}_{n+1}$ for all $n$, and, for each $n$, the finite subset $\{\varphi_j\}_{j=1}^n$ is an orthonormal basis of $\mathcal{V}_n$. We now define the $n$-th approximant to the entropy variable $u^\varepsilon$, to be given by the explicit formula 
\begin{equation}\label{eq:ansatz for galerkin}
    u^\varepsilon_n(t,x,\theta) := \sum_{j=1}^n \alpha_j^{\varepsilon,n}(t) \varphi_j(x,\theta), 
\end{equation}
where the coefficients $\alpha^{\varepsilon,n}_j$ are to be determined. This construction is such that, for large $n$, the function $u^\varepsilon_n$ ought to be a suitable approximation of $u^\varepsilon$; a weak solution in $H^1_\per(\Upsilon)$ of \eqref{eq:eps eqn}.

One can see, by plugging in $u^\varepsilon_n$ in place of $u^\varepsilon$ in \eqref{eq:eps eqn}, that the coefficients $\alpha^{\varepsilon,n}_j$ will be determined from requiring 
\begin{equation}\label{eq:galerkin how to get coeffs}
    (\varepsilon \id + A(u^\varepsilon_n)) \Big[\sum_{j=1}^n\frac{\de \alpha^{\varepsilon,n}_j}{\de t} \varphi_j \Big] =  \nabla_{\bxi} \cdot \Big( (\varepsilon I + \tilde M(u^\varepsilon_n)) \sum_{j=1}^n \alpha^{\varepsilon,n}_j \nabla_{\bxi} \varphi_j +  \tilde M(u^\varepsilon_n) V \Big) 
\end{equation}
in the weak sense in $\mathcal{V}_n$, \textit{i.e.}, for all $\phi \in \mathcal{V}_n$, there holds 
\begin{equation*}
  \begin{aligned}
      \sum_{j=1}^n \frac{\de \alpha^{\varepsilon,n}_j}{\de t} \int_\Upsilon (\varepsilon \id + & A(u^\varepsilon_n))  [\varphi_j] \phi \d \bxi \\ 
      &=  -\sum_{j=1}^n \alpha^{\varepsilon,n}_j\int_\Upsilon \nabla_{\bxi} \phi \cdot \left( (\varepsilon I + \tilde M(u^\varepsilon_n)) \nabla_{\bxi} \varphi_j +  \tilde M(u^\varepsilon_n) V \right) \d \bxi, 
  \end{aligned} 
\end{equation*}
for all $t \in (0,T)$, where we used linearity and $A(u^\varepsilon_n)[\frac{\de}{\de t}\alpha_j^{\varepsilon,n}(t)\varphi_j] = \frac{\de}{\de t}\alpha_j^{\varepsilon,n}(t)A(u^\varepsilon_n)[\varphi_j]$, which can be seen directly from the formula \eqref{eq:explicit Hessian}. Using the decomposition of $\phi$ into the basis functions $\{\varphi_k\}_{k=1}^n$, the previous relation is equivalent to requiring: for all $k\in\{1,\dots,n\}$, there holds 
\begin{equation*}
  \begin{aligned}
      \sum_{j=1}^n \frac{\de \alpha^{\varepsilon,n}_j}{\de t} \int_\Upsilon (\varepsilon \id + & A(u^\varepsilon_n))  [\varphi_j] \varphi_k \d \bxi \\ 
      &=  -\sum_{j=1}^n \alpha^{\varepsilon,n}_j\int_\Upsilon \nabla_{\bxi} \varphi_k \cdot \left( (\varepsilon I + \tilde M(u^\varepsilon_n)) \nabla_{\bxi} \varphi_j +  \tilde M(u^\varepsilon_n) V \right) \d \bxi, 
  \end{aligned} 
\end{equation*}
for all $t \in (0,T)$. Note that, provided the coefficients $\alpha^{\varepsilon,n}_j$ belong to $C^1([0,T])$, all terms in the previous equality are well-defined; in particular, the integrals are finite due to the boundedness of the functions $\varphi_j \in C^\infty_\per(\overline{\Upsilon})$. The above can be recast as a quasilinear system of ODEs of the form 
\begin{equation}\label{eq:ODE sys galerkin}
    \left\lbrace\begin{aligned}
    & \mathcal{M}(\gamma^{\varepsilon,n}(t)) \frac{\de \gamma^{\varepsilon,n}}{\de t}(t) = \mathcal{R}(\gamma^{\varepsilon,n}(t)), \\ 
    & \gamma^{\varepsilon,n}(0) = \gamma^{\varepsilon,n}_0, 
    \end{aligned}\right. 
\end{equation}
where we have defined the vector function $\gamma^{\varepsilon,n} = (\alpha^{\varepsilon,n}_1, \dots, \alpha^{\varepsilon,n}_n)$ and the entries of the matrices are given explicitly by 
\begin{equation}\label{eq:matrices for ODE sys galerkin}
    \begin{aligned}
        &\mathcal{M}(\gamma^{\varepsilon,n}(t))_{k j} = \int_\Upsilon (\varepsilon \id + A(u^\varepsilon_n))  [\varphi_j] \varphi_k \d \bxi, \\ 
        & \mathcal{R}(\gamma^{\varepsilon,n}(t))_{k j} = -\int_\Upsilon \nabla_{\bxi} \varphi_k \cdot \left( (\varepsilon I + \tilde M(u^\varepsilon_n)) \nabla_{\bxi} \varphi_j +  \tilde M(u^\varepsilon_n) V \right) \d \bxi, 
    \end{aligned}
\end{equation}
where $k,j$ range over all possible values in $\{1,\dots,n\}$, and the initial data is given explicitly by 
\begin{equation}\label{eq:initial data for galerkin sys}
    \alpha_j^{\varepsilon,n}(0) = \langle u^\varepsilon_0, \varphi_j \rangle_{H^2_\per(\Upsilon)}, 
\end{equation}
such that, with $u^\varepsilon_n(0) = u^\varepsilon_{n,0} \equiv \sum_{j=1}^n \alpha^{\varepsilon,n}_j(0) \varphi_j$, there holds the strong convergence 
\begin{equation}\label{eq:strong convergence initial data galerkin}
    u^{\varepsilon}_n(0) \to u^\varepsilon_0 \quad \text{in } H^2_\per(\Upsilon). 
\end{equation}

It follows from the positivity $\varepsilon>0$ and the positive semidefiniteness of the operator $A$ that the matrix $\mathcal{M}$ is invertible, and furthermore the product $\mathcal{M}^{-1}\mathcal{R}$ is locally Lipschitz; indeed, this quantity and its derivatives involve only exponential functions of the $\alpha^{\varepsilon,n}_j$ and their products, weighted by the reciprocal of a determinant which is bounded strictly away from zero for each $\varepsilon>0$. In turn, the Cauchy--Lipschitz Theorem for ODE systems implies the existence and uniqueness of $\gamma^{\varepsilon,n} \in C^1([0,T])$ solving \eqref{eq:ODE sys galerkin}.

It follows that we have proved the following result, which concludes this subsection. 

\begin{lemma}\label{lem:existence of galerkin coeffs}
    Let $\varepsilon \in (0,1)$, $n\in\mathbb{N}$, and $u^\varepsilon_{n,0} \in H^2_\per(\Upsilon)$ be given by 
    \begin{equation*}
        u^\varepsilon_{n,0} = \sum_{j=1}^n \varphi_j \langle u^\varepsilon_0, \varphi_j \rangle_{H^2_\per(\Upsilon)}. 
    \end{equation*}
    There exists $u^\varepsilon_n \in C^1(\overline{\Upsilon}_T)$ such that $u^\varepsilon_n(t,\cdot) \in \mathcal{V}_n$ for all $t \in [0,T]$, $u^\varepsilon_n(0) = u^\varepsilon_{n,0}$ and, for all $\psi \in \mathcal{V}_n$, there holds for all $t \in (0,T)$ 
    \begin{equation}\label{eq:galerkin regularised weak form}
  \int_{\Upsilon}(\varepsilon \id + A(u^\varepsilon_n)) [\partial_t u^{\varepsilon}_n] \psi \d \bxi =  - \int_\Upsilon \nabla_{\bxi}\psi \cdot \left( (\varepsilon I + \tilde M(u^\varepsilon_n)) \nabla_{\bxi} u^\varepsilon_n +  \tilde M(u^\varepsilon_n) V \right) \d \bxi, 
\end{equation}
\textit{i.e.}, in the weak sense with respect to $\mathcal{V}_n$, 
\begin{equation}\label{eq:galerkin regularised}
    (\varepsilon \id + A(u^\varepsilon_n)) [\partial_t u^{\varepsilon}_n] =  \nabla_{\bxi} \cdot \left( (\varepsilon I + \tilde M(u^\varepsilon_n)) \nabla_{\bxi} u^\varepsilon_n +  \tilde M(u^\varepsilon_n) V \right). 
\end{equation}
\end{lemma}

We then define $f^\varepsilon_n \in C^1([0,T]\times\overline{\Upsilon})$ and its angle-independent density $\rho^\varepsilon_n$ by 
\begin{equation}\label{eq:f eps n in terms of u eps n pre galerkin}
f^\varepsilon_n := (E^*)'[u^\varepsilon_n] = \frac{e^{u^\varepsilon_n}}{1+\int_0^{2\pi} e^{u^\varepsilon_n} \d \theta}, \quad \rho^\varepsilon_n := \int_0^{2\pi} f^\varepsilon_n \d \theta = \frac{\int_0^{2\pi} e^{u^\varepsilon_n} \d \theta}{1+\int_0^{2\pi} e^{u^\varepsilon_n} \d \theta}, 
\end{equation}
whence 
\begin{equation*}
   (\varepsilon \id + A(u^\varepsilon_n))[\partial_t u^\varepsilon_n] = \varepsilon \partial_t u^\varepsilon_n + \partial_t((E^*)'[u^\varepsilon_n]) = \varepsilon \partial_t u^\varepsilon_n + \partial_t f^\varepsilon_n,    
\end{equation*}
and from which we obtain, using also \eqref{eq:switching from duality bracket to entropy dissipation}, 
\begin{equation}\label{eq:ent diss rel}
   \int_\Upsilon (\varepsilon \id + A(u^\varepsilon_n))[\partial_t u^\varepsilon_n]u^\varepsilon_n \d \bxi = \frac{\varepsilon}{2} \frac{\de}{\de t} \int_\Upsilon |u^\varepsilon_n|^2 \d \bxi + \frac{\de}{\de t} E[f^\varepsilon_n]. 
\end{equation}
At this point, we record the relations 
\begin{equation}\label{eq:convenient expansion}
  \begin{aligned}
      &f^\varepsilon_n \partial_\theta u^\varepsilon_n = \partial_\theta f^\varepsilon_n = 2\sqrt{f^\varepsilon_n} \partial_\theta \sqrt{f^\varepsilon_n}, \\ 
      &(1-\rho^\varepsilon_n)f^\varepsilon_n \nabla u^\varepsilon_n = (1-\rho^\varepsilon_n) \nabla  f^\varepsilon_n  + f^\varepsilon_n \nabla \rho^\varepsilon_n = 2\sqrt{(1-\rho^\varepsilon_n)f^\varepsilon_n} \nabla \sqrt{(1-\rho^\varepsilon_n)f^\varepsilon_n} + 2 f^\varepsilon_n \nabla \rho^\varepsilon_n, 
  \end{aligned} 
\end{equation}
whence the equation \eqref{eq:galerkin regularised} may be rewritten as 
\begin{equation}\label{eq:galerkin reg better}
    \varepsilon \partial_t u^\varepsilon_n + \partial_t f^\varepsilon_n + \pe \dv((1-\rho^\varepsilon_n) f^\varepsilon_n \e(\theta)) = \varepsilon \Delta_{\bxi}u^\varepsilon_n + D_e \dv( (1-\rho^\varepsilon_n) \nabla f^\varepsilon_n + f^\varepsilon_n \nabla \rho^\varepsilon_n ) + \partial^2_\theta f^\varepsilon_n. 
\end{equation}
It follows from an integration by parts that \eqref{eq:weak form reg sys} is verified. We have therefore proved the existence part of Proposition \ref{prop:transformed reg sys existence and est}. It remains to obtain the uniform estimates, which is the focus of the next section. 

\subsection{Uniform estimates for the Galerkin approximations}\label{sec:unif est}

In this section, we perform the main entropy-dissipation estimate, and obtain bounds uniform in $n$ and $\varepsilon$ for both the entropy variable and the original unknown. After this, we establish similar estimates for their respective time derivatives. 

We first remark that, for all fixed $t,x$ there holds $\Vert f^\varepsilon_n (t,x,\cdot) \Vert_{L^1((0,2\pi))} \leq 1$ as well as $f^\varepsilon_n \geq 0$ and $0 \leq \rho^\varepsilon_n \leq 1$, 
from which we get 
\begin{equation}\label{eq:unif L1 on f eps n}
    \Vert f^\varepsilon_n \Vert_{L^\infty(\Omega_T;L^1(0,2\pi))} \leq 1. 
\end{equation}
This bound will be used in the proof of the main result of this subsection, which we now state. 

\begin{lemma}[Uniform estimates for Galerkin approximation]\label{lem:unif est for galerkin}
Let $\varepsilon,n,u^\varepsilon_n$ be as in Lemma \ref{lem:existence of galerkin coeffs} and $f^\varepsilon_n,\rho^\varepsilon_n$ be as in \eqref{eq:f eps n in terms of u eps n pre galerkin}. There exists a positive constant $C$, independent of $\varepsilon,n$ such that there holds the uniform estimates: 
    \begin{equation}\label{eq:main est gal i}
    \begin{aligned}
      \sup_{t \in [0,T]} E[f^\varepsilon_n](t) +  \sup_{t \in [0,T]}\int_\Upsilon |\sqrt{\varepsilon} u^\varepsilon_n(t)|^2 \d \bxi + \int_{\Upsilon_T} |\nabla_{\bxi} \sqrt{\varepsilon} u^\varepsilon_n|^2 \d \bxi \d t& \\ 
      + \int_{\Upsilon_T} (1-\rho^\varepsilon_n)|\nabla_{\bxi} \sqrt{f^\varepsilon_n} |^2  \d \bxi \d t + \int_{\Upsilon_T} |\nabla \sqrt{1-\rho^\varepsilon_n}|^2 \d x \d t& \\ 
      + \int_{\Upsilon_T} |\partial_\theta  \sqrt{f^\varepsilon_n}|^2 \d \bxi \d t + \int_{\Omega_T} |\nabla \rho^\varepsilon_n|^2  \d x \d t 
      &\leq C, 
        \end{aligned}
\end{equation}
and 
\begin{align}
     \Vert f^\varepsilon_n \Vert_{L^3(\Upsilon_T)} \leq C \label{eq:sqrt func di ben} \\ 
     \Vert \sqrt{(1-\rho^\varepsilon_n)f^\varepsilon_n} \Vert_{L^{\frac{10}{3}}(\Upsilon_T)} \leq C \label{eq:gradient of f sqrt 1- rho} \\ 
     \Vert f^\varepsilon_n \sqrt{1-\rho^\varepsilon_n} \Vert_{L^{\frac{6}{5}}(0,T;W^{1,\frac{6}{5}}(\Upsilon))} \leq C. \label{eq:bound on f sqrt 1-rho main one for limit}
\end{align}
\end{lemma}

\begin{proof}
\noindent 1. \textit{Entropy dissipation}: By testing \eqref{eq:galerkin regularised} with $u^\varepsilon_n$, which is an admissible test function since one can directly verify $u^{\varepsilon}_n, \partial_t u^\varepsilon_n \in \mathcal{V}_n$, we compute, using \eqref{eq:ent diss rel}, the entropy dissipation 
\begin{equation*}
    \begin{aligned}
        \frac{\varepsilon}{2} \frac{\de}{\de t} \int_\Upsilon |u^\varepsilon_n|^2 \d x \d \theta + \frac{\de}{\de t} E[f^\varepsilon_n] + \int_\Upsilon (\varepsilon I + \tilde{M}(u^\varepsilon_n) )\nabla_{\bxi} u^\varepsilon_n \cdot \nabla_{\bxi} u^\varepsilon_n \d \bxi = - \int_\Upsilon \tilde{M}(u^\varepsilon_n) V \cdot \nabla_{\bxi} u^\varepsilon_n \d \bxi. 
        \end{aligned}
\end{equation*}
Again, we emphasise that, since $\varphi_j \in C^\infty_\per(\overline{\Upsilon})$ for all $j$, the integrands in the previous line are composed of functions belonging to $C^1(\overline{\Upsilon}_T)$, whence the integrals are well-defined. Recall that the  matrix $\tilde{M}$ is positive semi-definite and symmetric. It follows that there exists a unique square-root matrix $\tilde{N}$ of the same dimension such that $\tilde{N}^T \tilde{N} = \tilde{M}$, namely 
\begin{equation*}
    \tilde{N}(u) = N(f[u]) = \left( \begin{matrix}
        \sqrt{D_e f(1-\rho)}I & 0 \\ 
        0 & \sqrt{f}
    \end{matrix} \right). 
\end{equation*}
In turn, there holds, for all $t \in [0,T]$, 
\begin{equation}\label{eq:penult step main dissip galerkin}
    \begin{aligned}
      E[f^\varepsilon_n](t) + \frac{\varepsilon}{2}  \int_\Upsilon &|u^\varepsilon_n(t)|^2 \d \bxi + \varepsilon \int_0^t \int_\Upsilon |\nabla_{\bxi} u^\varepsilon_n|^2  \d \bxi \d \tau + \int_0^t \int_\Upsilon \underbrace{\tilde{M}(u^\varepsilon_n) \nabla_{\bxi} u^\varepsilon_n \cdot \nabla_{\bxi} u^\varepsilon_n}_{\geq 0} \d \bxi \d \tau  \\ 
        &= E[f^\varepsilon_n](0) + \frac{\varepsilon}{2} \int_\Upsilon |u^\varepsilon_n(0)|^2 \d \bxi - \int_0^t \int_\Upsilon \tilde{N}(u^\varepsilon_n) \nabla_{\bxi} u^\varepsilon_n \cdot\tilde{N}(u^\varepsilon_n) V \d \bxi \d \tau, 
        \end{aligned}
\end{equation}
and observe that the final term on the right-hand side may be bounded by means of the Young inequality as follows 
\begin{equation}\label{eq:cauchy young penult galerkin}
   \begin{aligned}
       |\tilde{N}(u^\varepsilon_n) \nabla_{\bxi} u^\varepsilon_n \cdot\tilde{N}(u^\varepsilon_n) V| \leq& \frac{1}{2}|\tilde{N}(u^\varepsilon_n)\nabla_{\bxi} u^\varepsilon_n|^2 + \frac{1}{2}|\tilde{N}(u^\varepsilon_n) V|^2 \\ 
       =& \frac{1}{2}\tilde{M}(u^\varepsilon_n)\nabla_{\bxi} u^\varepsilon_n \cdot \nabla_{\bxi} u^\varepsilon_n + \frac{1}{2}\tilde{M}(u^\varepsilon_n)V \cdot V \\ 
       \leq & \frac{1}{2}\tilde{M}(u^\varepsilon_n)\nabla_{\bxi} u^\varepsilon_n \cdot \nabla_{\bxi} u^\varepsilon_n + \frac{1}{2} \Vert V \Vert_{L^\infty(\Upsilon_T)}^2|f^\varepsilon_n|, 
   \end{aligned} 
\end{equation}
where we used the boundedness of $\rho^\varepsilon_n$; \textit{cf.}~\eqref{eq:unif L1 on f eps n}. Recall from \eqref{eq:strong convergence initial data galerkin} that $u^\varepsilon_n(0) \to u^\varepsilon_0$ strongly in $H^2_\per(\Upsilon)$, and hence almost everywhere for a subsequence; which, with slight abuse of notation, we still label as $\{u^\varepsilon_n\}_{n}$. In view of the Morrey embedding $H^2(\Upsilon) \hookrightarrow C^{0,1/2}(\Upsilon)$, we deduce 
\begin{equation*}
    \Vert u^\varepsilon_n(0) - u^\varepsilon_0 \Vert_{L^\infty(\Upsilon)} \to 0, 
\end{equation*}
and hence $$\sup_n \Vert u^\varepsilon_n(0) \Vert_{L^\infty(\Upsilon)} < \infty.$$ 
It follows from the continuity of the exponential and the relations \eqref{eq:f eps n in terms of u eps n pre galerkin} that $f^\varepsilon_n(0) \to f^\varepsilon_0$ almost everywhere, and 
\begin{equation*}
    0 < f^\varepsilon_n(0) \leq \exp \big( \sup_n {\Vert u^\varepsilon_n(0) \Vert_{L^\infty(\Upsilon)}}\big) < \infty, 
\end{equation*}
and we remark that $ \exp(\sup_n{\Vert u^\varepsilon_n(0) \Vert_{L^\infty(\Upsilon)}})$ is constant and is therefore an integrable dominating function on the bounded domain $\Upsilon$. Using the previous estimate, an application of the Dominated Convergence Theorem yields, for all $q \in [1,\infty)$, 
\begin{equation}\label{eq:strong conv in Lq for initial f eps n}
    f^\varepsilon_n(0) \to f^\varepsilon_0 \quad \text{strongly in }L^q(\Upsilon) \text{ as } n \to \infty, 
\end{equation}
and a similar argument using also \eqref{eq:unif L1 on f eps n}, \textit{cf.}~\textit{e.g.}~Appendix \ref{sec:appendix initial data}, yields $\lim_{n \to \infty }E[f^\varepsilon_n(0)] = E[f^\varepsilon_0]$ in the sense of real numbers. It therefore follows, using also \eqref{eq:unif L1 on f eps n}, \eqref{eq:penult step main dissip galerkin}, and \eqref{eq:cauchy young penult galerkin}, that 
\begin{equation*}
    \begin{aligned}
      \sup_{t \in [0,T]} E[f^\varepsilon_n](t) +& \frac{\varepsilon}{2}  \bigg( \sup_{t \in [0,T]}\int_\Upsilon |u^\varepsilon_n(t)|^2 \d \bxi + \int_{\Upsilon_T} |\nabla_{\bxi} u^\varepsilon_n|^2 \d \bxi \d t \bigg) \\ 
      +& \frac{1}{2}\int_{\Upsilon_T} \tilde{M}(u^\varepsilon_n) \nabla_{\bxi} u^\varepsilon_n \cdot \nabla_{\bxi} u^\varepsilon_n \d \bxi \d t  \leq C \bigg( 1+ E[f^\varepsilon_0] + \varepsilon\int_\Upsilon |u^\varepsilon_0|^2 \d\bxi \bigg), 
        \end{aligned}
\end{equation*}
for some positive constant $C=C(\pe,D_e,T,\Upsilon)$ independent of $\varepsilon,n$. Our choice of regularised initial data (\textit{cf.}~Lemma \ref{lem:regularised initial data}) implies that the entire right-hand side of the above is bounded independently of $\varepsilon,n$. By rewriting the final term on the left-hand side in terms of $f^\varepsilon_n,\rho^\varepsilon_n$ using \eqref{eq:convenient expansion}, we get 
\begin{equation*}
    \begin{aligned}
      \sup_{t \in [0,T]} E[f^\varepsilon_n](t) +& \frac{\varepsilon}{2}  \bigg( \sup_{t \in [0,T]}\int_\Upsilon |u^\varepsilon_n(t)|^2 \d \bxi + \int_{\Upsilon_T} |\nabla_{\bxi} u^\varepsilon_n|^2 \d \bxi \d t \bigg) \\ 
      +& \frac{1}{2}\int_{\Upsilon_T} \Big( D_e f^\varepsilon_n(1-\rho^\varepsilon_n) \big|\nabla \log\big( \frac{f^\varepsilon_n}{1-\rho^\varepsilon_n}\big)\big|^2 + f^\varepsilon_n \big|\partial_\theta \log\big( \frac{f^\varepsilon_n}{1-\rho^\varepsilon_n}\big)\big|^2  \Big) \d \bxi \d t  \leq C, 
        \end{aligned}
\end{equation*}
\textit{i.e.}, 
\begin{equation}\label{eq:unif in n est for galerkin}
    \begin{aligned}
      \sup_{t \in [0,T]} E[f^\varepsilon_n](t) +  \sup_{t \in [0,T]}\int_\Upsilon |\sqrt{\varepsilon} u^\varepsilon_n(t)|^2 \d \bxi + \int_{\Upsilon_T} |\nabla_{\bxi} \sqrt{\varepsilon} u^\varepsilon_n|^2 \d \bxi \d t& \\ 
      + \int_{\Upsilon_T} (1-\rho^\varepsilon_n)|\nabla \sqrt{f^\varepsilon_n} |^2  \d \bxi \d t + \int_{\Omega_T} \rho^\varepsilon_n |\nabla \sqrt{1-\rho^\varepsilon_n}|^2 \d x \d t& \\ 
      + \int_{\Upsilon_T} |\partial_\theta  \sqrt{f^\varepsilon_n}|^2 \d \bxi \d t + \int_{\Omega_T} |\nabla \rho^\varepsilon_n|^2  \d x \d t 
      &\leq C, 
        \end{aligned}
\end{equation}
where we used the Fubini--Tonelli Theorem and the commutativity of the derivative in $x$ and integral in $\theta$ to rewrite $\int_{\Upsilon_T} \nabla f^\varepsilon_n \cdot \nabla \rho^\varepsilon_n \d \bxi \d t = \int_{\Omega_T} |\nabla \rho^\varepsilon_n|^2  \d x \d t$ for the final term on the left-hand side of \eqref{eq:unif in n est for galerkin}. Observe that, using the bounds in the previous estimate, 
\begin{equation*}
    \begin{aligned}
        \int_{\Omega_T} |\nabla \sqrt{1-\rho^\varepsilon_n}|^2 \d x \d t 
        = \frac{1}{4}\int_{\Omega_T} |\nabla \rho^\varepsilon_n|^2 \d x \d t + \int_{\Omega_T} \rho^\varepsilon_n |\nabla \sqrt{1-\rho^\varepsilon_n}|^2\d x \d t \leq C, 
    \end{aligned}
\end{equation*}
and similarly 
\begin{equation}\label{eq:partial bound i}
    \begin{aligned}
        \int_{\Upsilon_T} (1-\rho^\varepsilon_n)|\nabla_{\bxi} \sqrt{f^\varepsilon_n} |^2  \d \bxi \d t \leq \int_{\Upsilon_T} (1-\rho^\varepsilon_n)|\nabla \sqrt{f^\varepsilon_n} |^2  \d \bxi \d t + \int_{\Upsilon_T} |\partial_\theta \sqrt{f^\varepsilon_n} |^2  \d \bxi \d t \leq C. 
    \end{aligned}
\end{equation}
Collating these latter two bounds with \eqref{eq:unif in n est for galerkin} yields the estimate \eqref{eq:main est gal i}. 

\noindent 2. \textit{Interpolated bounds on $\sqrt{f}$ and $\sqrt{(1-\rho)f}$}: Using the product rule and the estimates \eqref{eq:partial bound i} and \eqref{eq:unif in n est for galerkin}, we find 
\begin{equation*}
  \begin{aligned}
      \int_{\Upsilon_T} |\nabla_{\bxi} \sqrt{(1-\rho^\varepsilon_n) f^\varepsilon_n} |^2  \d \bxi \d t     &\leq C\bigg(  \int_{\Upsilon_T} (1-\rho^\varepsilon_n)|\nabla_{\bxi} \sqrt{f^\varepsilon_n} |^2  \d \bxi \d t + \int_{\Omega_T} \rho^\varepsilon_n |\nabla \sqrt{1-\rho^\varepsilon_n}|^2 \d x \d t  \bigg) \\ 
      &\leq C. 
  \end{aligned} 
\end{equation*}
Meanwhile, the estimate \eqref{eq:unif L1 on f eps n} implies $\Vert \sqrt{(1-\rho^\varepsilon_n)f^\varepsilon_n}\Vert_{L^\infty(0,T;L^2(\Upsilon))} \leq C$, whence we deduce 
\begin{equation*}
    \Vert \sqrt{(1-\rho^\varepsilon_n)f^\varepsilon_n}\Vert_{L^\infty(0,T;L^2(\Upsilon))} + \Vert \sqrt{(1-\rho^\varepsilon_n)f^\varepsilon_n}\Vert_{L^2(0,T;H^1(\Upsilon))} \leq C. 
\end{equation*}
In turn, the classical interpolation lemma \cite[\S 1, Proposition 3.2]{DiBenedetto} implies 
\begin{equation*}
    \Vert \sqrt{(1-\rho^\varepsilon_n)f^\varepsilon_n} \Vert_{L^{\frac{10}{3}}(\Upsilon_T)} \leq C, 
\end{equation*}
and thus \eqref{eq:gradient of f sqrt 1- rho} is proved.

Our objective is then to make use of the \emph{linear} diffusion in angle to upgrade the integrability of $f^\varepsilon_n$; to do so, we use an adapted version of the aforementioned interpolation lemma to the space $L^\infty(\Omega_T;L^2(0,2\pi)) \cap L^2(\Omega_T;H^1(0,2\pi))$, \textit{cf.}~Appendix \ref{sec:appendix interpolation}. Using both \eqref{eq:unif L1 on f eps n} and the estimate on $\partial_\theta \sqrt{f^\varepsilon_n}$ from \eqref{eq:unif in n est for galerkin}, we remark that $$\Vert \sqrt{f^\varepsilon_n} \Vert_{L^\infty(\Omega_T;L^2(0,2\pi))} + \Vert \sqrt{f^\varepsilon_n} \Vert_{L^2(\Omega_T;H^1(0,2\pi))} \leq C.$$ A direct application of Lemma \ref{lem:our dibenedetto} then implies 
\begin{equation*}
    \Vert \sqrt{f^\varepsilon_n} \Vert_{L^6(\Upsilon_T)} \leq C, 
\end{equation*}
and \eqref{eq:sqrt func di ben} follows.

\noindent 3. \textit{Estimate on $f\sqrt{1-\rho}$}: We write 
\begin{equation*}
    \begin{aligned}
        \sqrt{1-\rho^\varepsilon_n} \nabla_{\bxi} f^\varepsilon_n = 2\sqrt{f^\varepsilon_n} \sqrt{1-\rho^\varepsilon_n} \nabla_{\bxi} \sqrt{f^\varepsilon_n},  
    \end{aligned}
\end{equation*}
whence H\"older's inequality yields 
\begin{equation*}
    \Vert \sqrt{1-\rho^\varepsilon_n} \nabla_{\bxi} f^\varepsilon_n \Vert_{L^{\frac{3}{2}}(\Upsilon_T)} \leq 2 \Vert \sqrt{f^\varepsilon_n} \Vert_{L^6(\Upsilon_T)} \Vert \sqrt{1-\rho^\varepsilon_n}\nabla_{\bxi} \sqrt{f^\varepsilon_n} \Vert_{L^2(\Upsilon_T)} \leq C, 
\end{equation*}
where we used \eqref{eq:main est gal i} and \eqref{eq:sqrt func di ben}. In turn, using the H\"older, Minkowski, and Jensen inequalities, we get 
\begin{equation*}
    \begin{aligned}
        \Vert \nabla_{\bxi}\big( f^\varepsilon_n \sqrt{1-\rho^\varepsilon_n} \big) \Vert_{L^{\frac{6}{5}}(\Upsilon_T)} &\leq \Vert  \sqrt{1-\rho^\varepsilon_n} \nabla_{\bxi} f^\varepsilon_n \Vert_{L^{\frac{6}{5}}(\Upsilon_T)} + \Vert f^\varepsilon_n \nabla \sqrt{1-\rho^\varepsilon_n} \Vert_{L^{\frac{6}{5}}(\Upsilon_T)} \\ 
        &\leq C\Vert \sqrt{1-\rho^\varepsilon_n} \nabla_{\bxi} f^\varepsilon_n \Vert_{L^{\frac{3}{2}}(\Upsilon_T)} + \Vert f^\varepsilon_n \Vert_{L^3(\Upsilon_T)} \Vert \nabla \sqrt{1-\rho^\varepsilon_n}\Vert_{L^2(\Omega_T)} \\ 
        &\leq C, 
    \end{aligned}
\end{equation*}
where we again used \eqref{eq:main est gal i} and \eqref{eq:sqrt func di ben}. Therefore, using again Jensen's inequality and \eqref{eq:sqrt func di ben} to bound $\Vert f^\varepsilon_n \sqrt{1-\rho^\varepsilon_n}\Vert_{L^{\frac{6}{5}}(\Upsilon_T)}$ uniformly in $n$ and $\varepsilon$, we obtain the estimate \eqref{eq:gradient of f sqrt 1- rho}. 
\end{proof}

In what follows, we estimate the time derivatives of the transformed variable and original unknowns; these bounds will then permit us to apply the Aubin--Lions Lemma in the proof of Theorem \ref{thm:main existence result}, \textit{cf.}~\S \ref{sec:proof of existence result}. To this end, we recall the angle-independent integrated entropy variable 
\begin{equation*}
    U^\varepsilon_n(t,x) = \int_0^{2\pi} u^\varepsilon_n(t,x,\theta) \d \theta, 
\end{equation*}
which was defined at the end of the statement of Proposition \ref{prop:transformed reg sys existence and est}. Observe that, by virtue of Jensen's inequality, the entropy estimate \eqref{eq:main est gal i} yields 
\begin{equation}\label{eq:est on U eps n}
    \sup_{t \in [0,T]} \int_\Omega |\sqrt{\varepsilon}U^\varepsilon_n(t)|^2 \d x + \int_{\Omega_T} |\nabla \sqrt{\varepsilon} U^\varepsilon_n|^2 \d x \d t \leq C. 
\end{equation}
Furthermore, integrating \eqref{eq:galerkin reg better} with respect to the angle variable, yields the evolution equation on the angle-independent quantities: 
\begin{equation}\label{eq:rho eqn in galerkin}
    \partial_t(\varepsilon U^\varepsilon_n + \rho^\varepsilon_n) + \pe \dv((1-\rho^\varepsilon_n)\p^\varepsilon_n) = \varepsilon \Delta U^\varepsilon_n + D_e \Delta \rho^\varepsilon_n 
\end{equation}
in the weak sense, where 
\begin{equation*}
    \p^\varepsilon_n(t,x) = \int_0^{2\pi} f^\varepsilon_n(t,x,\theta) \e(\theta) \d \theta 
\end{equation*}
satisfies, using the non-negativity of $f^\varepsilon_n$ and the boundedness of $\rho^\varepsilon_n$, the uniform estimate $$\Vert \p^\varepsilon_n \Vert_{L^\infty(\Omega_T)} \leq 1.$$

\begin{cor}[Uniform estimates on time derivatives]\label{cor:time deriv est gal}
  Let $\varepsilon,n,u^\varepsilon_n$ be as in Lemma \ref{lem:existence of galerkin coeffs} and $f^\varepsilon_n,\rho^\varepsilon_n$ be as in \eqref{eq:f eps n in terms of u eps n pre galerkin}. There exists a positive constant $C$, independent of $\varepsilon,n$, such that there holds the uniform estimates: 
    \begin{align}
        \Vert \partial_t(\varepsilon u^\varepsilon_n + f^\varepsilon_n) \Vert_{X'} \leq C, \label{eq:main time deriv i} \\ 
        \Vert \partial_t(\varepsilon U^\varepsilon_n + \rho^\varepsilon_n ) \Vert_{Y'} \leq C. \label{eq:main time deriv ii}
    \end{align}
\end{cor}

\begin{proof}

In view of the estimates of Lemma \ref{lem:unif est for galerkin}, a standard density argument shows that the weak formulation \eqref{eq:galerkin regularised weak form}, \textit{cf.}~\eqref{eq:galerkin reg better}, generalises to the space of test functions $\mathcal{A}$ defined in \eqref{eq:admissible test functions}, \textit{i.e.}, for all $\psi \in \mathcal{A}$, there holds 
\begin{equation}\label{eq:galerkin regularised weak form for time deriv est}
    \begin{aligned}
  - \langle \partial_t (\varepsilon u^\varepsilon_n &+ f^\varepsilon_n) , \psi \rangle + \pe\int_{\Upsilon_T} (1-\rho^\varepsilon_n)f^\varepsilon_n \e(\theta) \cdot \nabla \psi \d \bxi \d t \\ 
      =&\int_{\Upsilon_T} \Big( \varepsilon \nabla_{\bxi}u^\varepsilon_n \cdot \nabla_{\bxi}\psi + D_e \big((1-\rho^\varepsilon_n)\nabla f^\varepsilon_n + f^\varepsilon_n \nabla \rho^\varepsilon_n \big) \cdot \nabla \psi + \partial_\theta f^\varepsilon_n \partial_\theta \psi \Big) \d \bxi \d t. 
    \end{aligned}
\end{equation}

\noindent 1. \textit{Estimate on $\partial_t(\varepsilon u + f)$}: The weak formulation \eqref{eq:galerkin regularised weak form for time deriv est} implies that, for all $\psi \in \mathcal{A}$, there holds 
\begin{equation}\label{eq:first exp for time deriv est in galerkin}
  \begin{aligned}
     |\langle \partial_t(\varepsilon u^\varepsilon_n + f^\varepsilon_n), \psi \rangle| \leq& D_e \bigg| \int_{\Upsilon_T} (1-\rho^\varepsilon_n) \nabla f^\varepsilon_n \cdot \nabla \psi \d \bxi \d t \bigg| + D_e \bigg| \int_{\Upsilon_T} f^\varepsilon_n \nabla \rho^\varepsilon_n \cdot \nabla \psi \d \bxi \d t \bigg| \\ 
     &+ \sqrt{\varepsilon} \bigg| \int_{\Upsilon_T} \nabla_{\bxi}\sqrt{\varepsilon}u^\varepsilon_n \cdot \nabla_{\bxi}\psi \d \bxi \d t \bigg| + \bigg|\int_{\Upsilon_T} \partial_\theta f^\varepsilon_n \partial_\theta \psi \d \bxi \d t \bigg| \\ 
     &+ \pe \bigg| \int_{\Upsilon_T}  (1-\rho^\varepsilon_n)f^\varepsilon_n \e(\theta) \cdot \nabla \psi \d \bxi \d t \bigg| \\ 
     =:& \sum_{j=1}^5 I_j. 
  \end{aligned} 
\end{equation}

Using the estimates of Lemma \ref{lem:unif est for galerkin}, we bound each term on the right-hand side of \eqref{eq:first exp for time deriv est in galerkin} individually. To begin with, H\"older's inequality yields 
\begin{equation*}
    I_2 \leq \Vert f^\varepsilon_n \Vert_{L^3(\Upsilon_T)} \Vert \nabla \rho^\varepsilon_n \Vert_{L^2(\Upsilon_T)} \Vert \nabla_{\bxi} \psi \Vert_{L^{6}(\Upsilon_T)} \leq C\Vert \nabla_{\bxi} \psi \Vert_{L^{6}(\Upsilon_T)}. 
\end{equation*}
Next, we use the second relation in \eqref{eq:convenient expansion}, the previous bound, and Jensen's inequality to write 
\begin{equation*}
   \begin{aligned}
       I_1 \leq  C\Big(\Vert \sqrt{(1-\rho^\varepsilon_n)f^\varepsilon_n} \Vert_{L^{\frac{10}{3}}(\Upsilon_T)} \Vert \nabla_{\bxi} \sqrt{(1-\rho^\varepsilon_n)f^\varepsilon_n} \Vert_{L^{2}(\Upsilon_T)} \Vert \nabla_{\bxi} \psi \Vert_{L^{5}(\Upsilon_T)} \!+ \!I_2 \Big)\! \leq C \Vert \nabla_{\bxi} \psi \Vert_{L^{6}(\Upsilon_T)}. 
   \end{aligned} 
\end{equation*}
Similarly, using the first relation in \eqref{eq:convenient expansion}, 
\begin{equation*}
   \begin{aligned}
      I_4 \leq \Vert \sqrt{f^\varepsilon_n} \Vert_{L^6(\Upsilon_T)} \Vert \partial_\theta \sqrt{f^\varepsilon_n} \Vert_{L^2(\Upsilon_T)} \Vert \partial_\theta \psi \Vert_{L^{3}(\Upsilon_T)} \leq C \Vert \partial_\theta \psi \Vert_{L^{3}(\Upsilon_T)}. 
   \end{aligned} 
\end{equation*}
For the third term, we have 
\begin{equation*}
    I_3 \leq \sqrt{\varepsilon} \Vert \nabla_{\bxi} \sqrt{\varepsilon} u^\varepsilon_n \Vert_{L^2(\Upsilon_T)} \Vert \nabla_{\bxi} \psi \Vert_{L^2(\Upsilon)} \leq C \sqrt{\varepsilon}\Vert \nabla_{\bxi} \psi \Vert_{L^2(\Upsilon)}. 
\end{equation*}
For the final integral, the uniform boundedness of $\rho^\varepsilon_n$ and Jensen's inequality implies 
\begin{equation*}
    \begin{aligned}
        I_5 \leq & \Vert f^\varepsilon_n \Vert_{L^{3}(\Upsilon_T)} \Vert \nabla_{\bxi} \psi \Vert_{L^{\frac{3}{2}}(\Upsilon_T)} \leq C \Vert \nabla_{\bxi} \psi \Vert_{L^{\frac{3}{2}}(\Upsilon_T)}. 
    \end{aligned}
\end{equation*}
By combining the previous bounds and using Jensen's inequality, we deduce 
\begin{equation*}
    |\langle \partial_t(\varepsilon u^\varepsilon_n + f^\varepsilon_n) , \psi \rangle| \leq C \Vert \nabla_{\bxi} \psi \Vert_{L^6(\Upsilon_T)}, 
\end{equation*}
and \eqref{eq:main time deriv i} follows. 

\noindent 2. \textit{Estimate on $\partial_t (\varepsilon U + \rho)$}: Testing \eqref{eq:rho eqn in galerkin} against $\psi \in \mathcal{A}_s$, we obtain 
\begin{equation*}
    \begin{aligned}
        \langle \partial_t(\varepsilon U^\varepsilon_n + \rho^\varepsilon_n) , \psi \rangle = &\pe\int_{\Omega_T} (1-\rho^\varepsilon_n) \p^\varepsilon_n \cdot \nabla \psi \d x \d t \\ 
        &- \sqrt{\varepsilon} \int_{\Omega_T} \nabla \sqrt{\varepsilon}U^\varepsilon_n \cdot \nabla \psi \d x \d t - D_e \int_{\Omega_T} \nabla \rho^\varepsilon_n \cdot \nabla \psi \d x \d t, 
    \end{aligned}
\end{equation*}
from which we deduce, using \eqref{eq:est on U eps n}, the uniform boundedness of $\p^\varepsilon_n$, and Jensen's inequality, 
\begin{equation*}
    \begin{aligned}
        |\langle \partial_t(\varepsilon U^\varepsilon_n &+ \rho^\varepsilon_n) , \psi \rangle| \\ 
        \leq & C \Big( \Vert \nabla \psi \Vert_{L^1(\Omega_T)} + \sqrt{\varepsilon} \Vert \nabla \sqrt{\varepsilon} U^\varepsilon_n \Vert_{L^2(\Omega_T)} \Vert \nabla \psi \Vert_{L^2(\Omega_T)}  + \Vert \nabla \rho^\varepsilon_n \Vert_{L^2(\Omega_T)}\Vert \nabla \psi \Vert_{L^2(\Omega_T)} \Big) \\ 
        \leq & C \Vert \nabla \psi \Vert_{L^2(\Omega_T)} , 
    \end{aligned}
\end{equation*}
whence we deduce \eqref{eq:main time deriv ii}. 
\end{proof}

The results of Lemmas \ref{lem:existence of galerkin coeffs}-\ref{lem:unif est for galerkin} and Corollary \ref{cor:time deriv est gal} prove Proposition \ref{prop:transformed reg sys existence and est}.

\section{Proof of the Main Existence Theorem}\label{sec:proof of existence result}

In this section, we study the convergence of the approximate solutions and obtain the existence of a weak solution to \eqref{eq: model 4}. We shall essentially take the double-limit as $\varepsilon \to 0$ in the regularisation parameter and $n\to\infty$ in the Galerkin approximation simultaneously, which will entail the use of a diagonal argument. We explain why we proceed in this manner in the next paragraph. 

The reason for employing a double-limit strategy is that the time-derivative estimates of Corollary \ref{cor:time deriv est gal} do not decouple into separate estimates on $\partial_t \varepsilon U^\varepsilon_n$ and $\partial_t \rho^\varepsilon_n$ individually. This means that the Aubin--Lions Lemma yields the strong convergence of only the sum $\varepsilon U^\varepsilon_n + \rho^\varepsilon_n$ as $n\to\infty$ for fixed $\varepsilon$. It is then, to the authors' knowledge, not possible to obtain the strong convergence $\rho^\varepsilon_n \to \rho^\varepsilon$ for fixed $\varepsilon$ without having already determined that $\varepsilon U^\varepsilon_n$ converges strongly. While the classical theory of \cite{ladyzhenskaya} implies that the linear heat term $\varepsilon(\partial_t u^\varepsilon_n - \Delta_{\bxi} u^\varepsilon_n)$ ought to yield $H^2$-regularity of $u^\varepsilon_n$ and hence $L^2$-boundedness of $\partial_t u^\varepsilon_n$ (from which one could infer strong convergence of $u^\varepsilon_n$ by another application of Aubin--Lions), the presence of the non-local operator $A$ and the non-local matrix $\tilde{M}$ in the equation \eqref{eq:galerkin regularised} makes \emph{quantitative estimates} rather intractable. In turn, we are not able to deduce that $U^\varepsilon_n$ converges strongly as $n\to\infty$ while keeping $\varepsilon$ fixed. Instead, we rely on the uniform bound on $\Vert \sqrt{\varepsilon}U^\varepsilon_n \Vert_{L^2(\Omega_T)}$ to deduce that $\varepsilon U^\varepsilon_n \to 0$ strongly as $\varepsilon\to0$; whence the aforementioned strong convergence of the sum $\varepsilon U^\varepsilon_n + \rho^\varepsilon_n$ yields the desired strong convergence of $\rho^\varepsilon_n$. 

The main difficulty in implementing this strategy lies in making sure that the sequence of regularised initial data converges appropriately. This is all-the-more subtle since the initial functions $\{f^\varepsilon_n(0),\rho^\varepsilon_n(0)\}_{\varepsilon,n}$ are determined via the usual non-linear transformation from $u^\varepsilon_n(0)$, the $n$-th Galerkin approximation of the initial data $u^\varepsilon_0$, \textit{i.e.}, 
\begin{equation}\label{eq:initial values for f eps and co}
    u^\varepsilon_n(0) = \sum_{j=1}^{n} \langle u^\varepsilon_0,\varphi_j \rangle_{H^2_\per(\Upsilon)} \varphi_j, \quad f^\varepsilon_n(0) = \frac{e^{u^\varepsilon_n(0)}}{1+\int_0^{2\pi}e^{u^\varepsilon_n(0)} \d \theta}, \quad \rho^\varepsilon_n(0) = \int_0^{2\pi} f^\varepsilon_n(0) \d \theta, 
\end{equation}
where we recall that $\{\varphi_j \}_j$ is the orthonormal basis of $H^2_\per(\Upsilon)$ used in \S \ref{sec:existence galerkin}. 

The next lemma shows that there exists a choice of diagonal subsequence from the array $\{f^\varepsilon_n\}_{\varepsilon,n}$ which preserves the desired convergence to the initial data $f_0$. The proof is delayed to Appendix \ref{sec:appendix initial data}.

\begin{lemma}[Diagonal convergence of initial data]\label{lem:convergence of initial data}
There exists a diagonal subsequence of the array $\{f^\varepsilon_n\}_{\varepsilon,n}$, which, with slight abuse of notation, we denote by $\{f^\varepsilon\}_\varepsilon$, such that 
\begin{equation*}
    \lim_{\varepsilon \to 0}\Vert f^\varepsilon(0) - f_0 \Vert_{L^p(\Upsilon)} = 0. 
\end{equation*}
\end{lemma}

We are now ready to give the proof of the main existence theorem; Theorem \ref{thm:main existence result}. We will show that the sequence $\{f^\varepsilon\}_\varepsilon$ constructed in the previous lemma admits a subsequential limit which is indeed a weak solution of \eqref{eq: model 4} in the sense of Definition \ref{def:weak sol}. We emphasise that $\{f^\varepsilon\}_\varepsilon$ satisfies the uniform estimates of Proposition \ref{prop:transformed reg sys existence and est}. 

\begin{proof}[Proof of Theorem \ref{thm:main existence result}]

We assume for the time being that the test function $\psi$ belongs to the stronger space $\mathcal{A}$. We shall relax this assumption at the end of the proof. 

\noindent 1. \textit{Convergence of time-derivative}: The uniform bound \eqref{eq:main time deriv reg i} implies $$\partial_t (\varepsilon u^\varepsilon + f^\varepsilon) \overset{*}{\rightharpoonup} \zeta \quad \text{ weakly-* in } X', $$
for some unkown $\zeta \in X'$. The estimate \eqref{eq:sqrt func di ben reg} implies $f^\varepsilon \rightharpoonup f$ weakly in $L^3(\Upsilon_T)$ whence, for all $\psi \in \mathcal{A}$, 
\begin{equation*}
\begin{aligned}
    \langle \zeta, \psi \rangle = \lim_{\varepsilon \to 0} \langle \partial_t(\varepsilon u^\varepsilon + f^\varepsilon) , \psi \rangle &= -\lim_{\varepsilon\to0}\sqrt{\varepsilon}\int_{\Upsilon_T} \sqrt{\varepsilon}u^\varepsilon \partial_t \psi \d\bxi \d t -\lim_{\varepsilon\to0}\int_{\Upsilon_T} f^\varepsilon \partial_t \psi \d \bxi \d t \\ 
    &= -\int_{\Upsilon_T} f \partial_t \psi \d \bxi \d t, 
\end{aligned}
\end{equation*}
where we used that $\Vert \sqrt{\varepsilon} u^\varepsilon\Vert_{L^2(\Upsilon_T)} \leq C$ from \eqref{eq:main est reg i} to obtain the final line. It therefore follows that $\zeta = \partial_t f$ and $\partial_t f^\varepsilon \overset{*}{\rightharpoonup} \partial_t f$ weakly-* in $X'$. An analogous argument shows $\partial_t \rho^\varepsilon \overset{*}{\rightharpoonup} \partial_t \rho$ weakly-* in $Y'$.

\noindent 2. \textit{Strong convergence of $\rho^\varepsilon$}: To begin with, \eqref{eq:main est reg i} implies $\Vert \rho^\varepsilon \Vert_{L^2(0,T;H^1(\Omega))} \leq C$, whence it follows that $\rho^\varepsilon \rightharpoonup \rho$ weakly $L^2(0,T;H^1(\Omega))$. Meanwhile, the uniform estimates from Proposition \ref{prop:transformed reg sys existence and est} on $\sqrt{\varepsilon} U^\varepsilon$ and $\rho^\varepsilon$ imply that $$\Vert \varepsilon U^\varepsilon + \rho^\varepsilon \Vert_{L^\infty(0,T;L^2(\Omega))} + \Vert \varepsilon U^\varepsilon + \rho^\varepsilon \Vert_{L^2(0,T;H^1(\Omega))} \leq C.$$
In turn, the time-derivative bound \eqref{eq:main time deriv reg ii} and the Aubin--Lions Lemma imply that, up to a subsequence which we do not relabel, 
\begin{equation*}
    \varepsilon U^\varepsilon + \rho^\varepsilon \to A \quad \text{strongly in } L^2(\Omega_T). 
\end{equation*}
for some unknown $A \in L^2(\Omega_T)$. Note that the triangle inequality and the uniform estimate on $\sqrt{\varepsilon}U^\varepsilon$ from Proposition \ref{prop:transformed reg sys existence and est} imply 
\begin{equation*}
   \begin{aligned}
       \Vert \rho^\varepsilon - A \Vert_{L^2(\Omega_T)}  \leq & \Vert \varepsilon U^\varepsilon + \rho^\varepsilon - A \Vert_{L^2(\Omega_T)}  + \sqrt{\varepsilon} \underbrace{\Vert \sqrt{\varepsilon} U^\varepsilon \Vert_{L^2(\Omega_T)}}_{\leq C} \to 0, 
   \end{aligned}
\end{equation*}
and thus by the uniqueness of limits, $A=\rho$, \textit{i.e.}, 
\begin{equation}
    \rho^\varepsilon \to \rho \quad \text{strongly in } L^2(\Omega_T), 
\end{equation}
and $\rho \in [0,1]$ a.e.~in $\Omega_T$. Moreover, for all $q \in (2,\infty)$, the boundedness of $\rho,\rho^\varepsilon$ implies 
\begin{equation*}
    \begin{aligned}
        \int_{\Omega_T} |\rho^\varepsilon - \rho|^q \d x \d t = \int_{\Omega_T} |\rho^\varepsilon - \rho|^{q-2} |\rho^\varepsilon - \rho|^2 \d x \d t \leq \int_{\Omega_T} |\rho^\varepsilon - \rho|^2 \d x \d t \to 0. 
    \end{aligned}
\end{equation*}
Using Jensen's inequality to cover the indices $q \in [1,2)$, it follows that, for all $q \in [1,\infty)$, 
\begin{equation}\label{eq:strong conv rho}
    \rho^\varepsilon \to \rho \quad \text{strongly in } L^q(\Omega_T). 
\end{equation}

The previous convergence implies that $\rho^\varepsilon \to \rho$ a.e.~in $\Omega_T$; up to a subsequence which we do not relabel. In turn, given any $q \in [1,\infty)$, an application of the Dominated Convergence Theorem shows, using again $0 \leq \rho,\rho^\varepsilon_n \leq 1$ to establish an admissible constant majorant function, 
\begin{equation*}
    \begin{aligned}
        \int_{\Omega_T} \big|\sqrt{1-\rho^\varepsilon} - \sqrt{1-\rho} \big|^q \d x \d t \to 0, 
    \end{aligned}
\end{equation*}
\textit{i.e.}, for all $q \in [1,\infty)$, 
\begin{equation}\label{eq:strong cov for sqrt 1-rho}
    \sqrt{1-\rho^\varepsilon} \to \sqrt{1-\rho} \quad \text{strongly in } L^q(\Omega_T). 
\end{equation}

\noindent 3. \textit{Weak convergences and bounds}: The uniform estimate \eqref{eq:bound on f sqrt 1-rho main one for limit reg}  implies that there exists $B \in L^{\frac{6}{5}}(0,T;W^{1,\frac{6}{5}}(\Upsilon))$ such that, up to a subsequence which we do not relabel, there holds 
\begin{equation*}
    f^\varepsilon \sqrt{1-\rho^\varepsilon} \rightharpoonup B \quad \text{weakly in } L^{\frac{6}{5}}(0,T;W^{1,\frac{6}{5}}(\Upsilon)). 
\end{equation*}
Meanwhile, the uniform estimate \eqref{eq:sqrt func di ben reg}  implies that $f^\varepsilon \rightharpoonup f$ weakly in $L^3(\Upsilon_T)$, so that, recalling the strong convergence \eqref{eq:strong cov for sqrt 1-rho}, there holds 
\begin{equation*}
    f^\varepsilon \sqrt{1-\rho^\varepsilon} \rightharpoonup f \sqrt{1-\rho} \quad \text{weakly in } L^3(\Upsilon_T), 
\end{equation*}
and the uniqueness of weak limits implies $B = f \sqrt{1-\rho}$ a.e.~in $\Upsilon_T$. We therefore have 
\begin{equation}\label{eq:weak conv for f sqrt 1-rho}
    f^\varepsilon \sqrt{1-\rho^\varepsilon} \rightharpoonup f\sqrt{1-\rho} \quad \text{weakly in } L^{\frac{6}{5}}(0,T;W^{1,\frac{6}{5}}(\Upsilon)). 
\end{equation}
In turn, recalling the strong convergence \eqref{eq:strong cov for sqrt 1-rho}, we have that 
\begin{equation}\label{eq:weak conv for diffusive term i}
    \sqrt{1-\rho^\varepsilon} \nabla_{\bxi}(f^\varepsilon \sqrt{1-\rho^\varepsilon}) \rightharpoonup \sqrt{1-\rho} \nabla_{\bxi} (f\sqrt{1-\rho}) \quad \text{weakly in } L^{\frac{6}{5}}(\Upsilon_T). 
\end{equation}
Analogous arguments show $\sqrt{1-\rho^\varepsilon}\nabla \sqrt{f^\varepsilon} \rightharpoonup \sqrt{1-\rho}\nabla\sqrt{f}$ and $\nabla \sqrt{1-\rho^\varepsilon} \rightharpoonup \nabla \sqrt{1-\rho}$ weakly in $L^2(\Omega_T)$. 

Similarly, we deduce from \eqref{eq:main est reg i} and \eqref{eq:sqrt func di ben reg} that 
\begin{equation*}
    \Vert \partial_\theta f^\varepsilon \Vert_{L^{\frac{3}{2}}(\Upsilon_T)} \leq C\Vert \sqrt{f^\varepsilon} \Vert_{L^6(\Upsilon_T)} \Vert \partial_\theta \sqrt{f^\varepsilon} \Vert_{L^2(\Upsilon_T)} \leq C, 
\end{equation*}
whence $f^\varepsilon$ is uniformly bounded in $L^{\frac{3}{2}}(\Omega_T;W^{1,\frac{3}{2}}(0,2\pi))$. Using the uniqueness of limits and $f^\varepsilon \rightharpoonup f$ weakly in $L^3(\Upsilon_T)$, we deduce as before that 
\begin{equation}\label{eq:weak conv of dtheta f eps}
    \partial_\theta f^\varepsilon \rightharpoonup \partial_\theta f \quad \text{weakly in } L^{\frac{3}{2}}(\Upsilon_T). 
\end{equation}
An analogous approach shows $\partial_\theta \sqrt{f^\varepsilon} \rightharpoonup \partial_\theta \sqrt{f}$ weakly in $L^2(\Upsilon_T)$. 

In view of the weak (resp.~weak-*) lower semicontinuity of the norms, the limits satisfy 
\begin{equation}\label{eq:bounds final weak lower semi}
    \begin{aligned}
       &f \in L^3(\Upsilon_T), \quad  \partial_t f \in X' , \quad \partial_\theta \sqrt{f} \in L^{2}(\Upsilon_T), \quad \nabla\sqrt{1-\rho} \in L^2(\Omega_T), \\ 
       &\partial_t \rho \in Y', \quad \sqrt{1-\rho}\nabla \sqrt{f} \in L^2(\Upsilon_T), \quad f\sqrt{1-\rho} \in L^{\frac{6}{5}}(0,T;W^{1,\frac{6}{5}}(\Upsilon)). 
    \end{aligned}
\end{equation}

\noindent 4. \textit{Convergence of spatial diffusions}: We rewrite the diffusive terms as follows 
\begin{equation}\label{eq:rewriting of diffusive terms later}
    \begin{aligned}
        (1-\rho^\varepsilon)\nabla f^\varepsilon + f^\varepsilon \nabla \rho^\varepsilon = \sqrt{1-\rho^\varepsilon} \nabla (f^\varepsilon \sqrt{1-\rho^\varepsilon}) + \frac{3}{2} f^\varepsilon \nabla \rho^\varepsilon. 
    \end{aligned}
\end{equation}
The weak convergence \eqref{eq:weak conv for diffusive term i} implies 
\begin{equation}\label{eq:weak conv diff i actual test func}
   \lim_{\varepsilon \to 0} \int_{\Upsilon_T} \sqrt{1-\rho^\varepsilon} \nabla (f^\varepsilon \sqrt{1-\rho^\varepsilon}) \cdot \nabla  \psi \d \bxi \d t = \int_{\Upsilon_T} \sqrt{1-\rho} \nabla (f\sqrt{1-\rho}) \cdot \nabla \psi \d \bxi \d t, 
\end{equation}
and the desired convergence for the first term on the right-hand side of \eqref{eq:rewriting of diffusive terms later} is verified.

The second term on the right-hand side of \eqref{eq:rewriting of diffusive terms later} is dealt with as follows; we essentially use an integration by parts to place two derivatives on the test function prior to passing to the limit. We write 
\begin{equation*}
    f^\varepsilon \nabla \rho^\varepsilon = - 2 f^\varepsilon \sqrt{1-\rho^\varepsilon} \nabla \sqrt{1-\rho^\varepsilon}, 
\end{equation*}
and observe that an integration by parts yields 
\begin{equation}\label{eq:steps to decompose final convergence}
    \begin{aligned}
        \int_{\Upsilon_T} f^\varepsilon \nabla \rho^\varepsilon \cdot \nabla \psi \d \bxi \d t =& -2 \int_{\Upsilon_T} f^\varepsilon \sqrt{1-\rho^\varepsilon} \nabla \sqrt{1-\rho^\varepsilon} \cdot \nabla \psi \d \bxi \d t \\ 
        =& 2 \int_{\Upsilon_T}  \sqrt{1-\rho^\varepsilon} \dv \big( f^\varepsilon \sqrt{1-\rho^\varepsilon} \nabla \psi \big) \d \bxi \d t \\ 
        =& I_1^\varepsilon + I_2^\varepsilon, 
    \end{aligned}
\end{equation}
where 
\begin{equation*}
    \begin{aligned}
        I_1^\varepsilon :=& 2 \int_0^T \int_\Upsilon  \sqrt{1-\rho^\varepsilon} \nabla \big( f^\varepsilon \sqrt{1-\rho^\varepsilon} \big) \cdot \nabla \psi \d \bxi \d t, \\ 
        I_2^\varepsilon :=& 2\int_0^T \int_\Upsilon  \sqrt{1-\rho^\varepsilon} \big( f^\varepsilon \sqrt{1-\rho^\varepsilon}\big) \Delta \psi  \d \bxi \d t. 
    \end{aligned}
\end{equation*}
The weak convergence from \eqref{eq:weak conv for f sqrt 1-rho} along with the strong convergence from \eqref{eq:strong cov for sqrt 1-rho} implies that there holds the weak convergences in $L^{1}(\Upsilon_T)$: 
\begin{equation*}
    \begin{aligned}
        & \sqrt{1-\rho^\varepsilon} \nabla \big( f^\varepsilon \sqrt{1-\rho^\varepsilon} \big) \rightharpoonup \sqrt{1-\rho} \nabla \big( f \sqrt{1-\rho} \big), \\ 
        & \sqrt{1-\rho^\varepsilon}  \big( f^\varepsilon \sqrt{1-\rho^\varepsilon} \big) \rightharpoonup \sqrt{1-\rho}  \big( f \sqrt{1-\rho} \big), 
    \end{aligned}
\end{equation*}
from which we deduce 
\begin{equation*}
    \begin{aligned}
       \lim_{\varepsilon \to 0} I_1^\varepsilon = & 2 \int_{\Upsilon_T}  \sqrt{1-\rho} \nabla \big( f \sqrt{1-\rho} \big) \cdot \nabla \psi \d \bxi \d t, \\ 
       \lim_{\varepsilon \to 0} I_2^\varepsilon = & 2\int_{\Upsilon_T}  \sqrt{1-\rho} \big( f \sqrt{1-\rho}\big) \Delta \psi  \d \bxi \d t. 
    \end{aligned}
\end{equation*}
Using the computation in \eqref{eq:steps to decompose final convergence} again, we obtain 
\begin{equation*}
    \begin{aligned}
        \lim_{\varepsilon\to 0}\int_{\Upsilon_T} f^\varepsilon \nabla \rho^\varepsilon \cdot \nabla \psi \d \bxi \d t =& 2 \int_{\Upsilon_T}  \sqrt{1-\rho} \dv \big( f \sqrt{1-\rho} \nabla \psi \big) \d \bxi \d t \\ 
        =& \int_{\Upsilon_T} f \nabla \rho \cdot \nabla \psi \d \bxi \d t, 
    \end{aligned}
\end{equation*}
where we used that the above quantities are regular enough to perform the integration by parts; \textit{cf.}~\eqref{eq:bounds final weak lower semi}. By collating the previous limit with \eqref{eq:weak conv diff i actual test func} and using \eqref{eq:rewriting of diffusive terms later}, we get 
\begin{equation*}
    \lim_{\varepsilon \to 0} \int_{\Upsilon_T} \big( (1-\rho^\varepsilon)\nabla f^\varepsilon + f^\varepsilon \nabla \rho^\varepsilon \big) \cdot \nabla \psi \d \bxi \d t = \int_{\Upsilon_T} \big( (1-\rho)\nabla f + f \nabla \rho \big) \cdot \nabla \psi \d \bxi \d t. 
\end{equation*}

\noindent 5. \textit{Convergence of angular diffusion, non-local drift, and $\varepsilon$-error term}: For the angular diffusion, the weak convergence of $\partial_\theta f^\varepsilon$ in \eqref{eq:weak conv of dtheta f eps} implies directly 
\begin{equation*}
   \begin{aligned}
      \lim_{\varepsilon \to 0} \int_{\Upsilon_T} \partial_\theta f^\varepsilon \partial_\theta \psi \d \bxi \d t = \int_{\Upsilon_T} \partial_\theta f \partial_\theta \psi \d\bxi \d t. 
   \end{aligned} 
\end{equation*}

For the non-local drift, the strong convergence of $\rho^\varepsilon$ from \eqref{eq:strong conv rho} and the weak convergence of $f^\varepsilon$ implies $(1-\rho^\varepsilon)f^\varepsilon \rightharpoonup (1-\rho) f$ weakly in $L^1(\Upsilon_T)$, whence 
\begin{equation*}
    \lim_{\varepsilon\to0}\int_{\Upsilon_T} (1-\rho^\varepsilon) f^\varepsilon \e(\theta) \cdot \nabla \psi \d \bxi \d t = \int_{\Upsilon_T} (1-\rho^\varepsilon) f^\varepsilon \e(\theta) \cdot \nabla \psi \d \bxi \d t. 
\end{equation*}

Similarly, we have 
\begin{equation*}
    \lim_{\varepsilon \to 0}\int_\Upsilon \varepsilon \nabla_{\bxi}u^\varepsilon \cdot \nabla_{\bxi} \psi \d \bxi \d t \leq \lim_{\varepsilon \to 0}\sqrt{\varepsilon}\underbrace{\Vert \nabla_{\bxi} \sqrt{\varepsilon} u^\varepsilon \Vert_{L^2(\Upsilon)}}_{\leq C} \Vert \nabla_{\bxi} \psi \Vert_{L^2(\Upsilon)} = 0, 
\end{equation*}
where we used the uniform estimate \eqref{eq:main est reg i}.

\noindent 6. \textit{Admissible test functions and attainment of initial data}: It therefore follows that the weak formulation of Definition \ref{def:weak sol} holds for all test functions $\psi \in \mathcal{A}$. Using the density of this latter space in $X$ and the fact that $f,\rho$ and their derivatives are bounded in suitable spaces, \textit{cf.}~equation \eqref{eq:weak form is well defined in intro}, the weak formulation extends to test functions belonging to $X$. 

Finally, the attainment of the initial data in the dual space $Z'$ is clear from Remark \ref{rem:ibp for duality bracket} and Lemma \ref{lem:convergence of initial data}. The proof is complete. 
\end{proof}

We end this section by recording that the solution obtained in Theorem \ref{thm:main existence result} also satisfies an alternative notion of weak solution, which does not require the test functions of the weak formulation to be periodic with respect to the space-angle variable $\bxi$; this is the content of the next lemma. This formulation is convenient when implementing the method of De Giorgi to study the regularity of the solution (\textit{cf}.~\cite{regularity}), which will be the subject of future works. The proof is identical to that of Theorem \ref{thm:main existence result}, using that integrating by parts over all of $\mathbb{R}^3$ against an arbitrary non-periodic test function $\psi \in C^\infty_c((0,T)\times\mathbb{R}^3)$ yields no boundary terms, as was also the case when testing against periodic test functions $\psi \in X$ and integrating over one periodic cell. 

\begin{lemma}
    Let $f \in \mathcal{X}$ be the weak solution obtained in Theorem \ref{thm:main existence result}. Then, for all $\psi \in C^\infty_c((0,T)\times\mathbb{R}^3)$, there holds 
      \begin{equation*}
    \begin{aligned}
      \int_0^T \int_{\mathbb{R}^3} f \partial_t \psi \d \bxi \d t + & \pe \int_{0}^T \int_{\mathbb{R}^3} (1-\rho) f \e(\theta) \cdot \nabla \psi \d \bxi \d t \\ 
      &=\int_{0}^T \int_{\mathbb{R}^3} \Big( D_e\big((1-\rho)\nabla f + f \nabla \rho \big) \cdot \nabla \psi + \partial_\theta f\partial_\theta \psi \Big) \d \bxi \d t. 
    \end{aligned}
\end{equation*}
\end{lemma}

\section{Uniqueness for Null P\'eclet Number}\label{sec:uniqueness zero pe}

In the case $\pe=0$, the equation reads 
\begin{equation}\label{eq: model 4 zero pe}
    \left\lbrace \begin{aligned}
    & \partial_t f = D_e \dv ((1-\rho)\nabla f + f \nabla \rho) + \partial^2_\theta f, \\ 
    &\partial_t \rho = D_e \Delta \rho. 
    \end{aligned}\right. 
\end{equation}
The main observation is that, as $\rho$ satisfies the heat equation, uniqueness for $\rho$ is trivial. With this information, we employ the method of Gajewski \cite{gajewski} to deduce that the solution of \eqref{eq: model 4 zero pe} is unique within the class $\mathcal{X}$; this approach had been used for a similar uniqueness argument by J\"ungel and Zamponi in \cite{JungelZamponi}. 

\begin{proof}[Proof of Theorem \ref{thm:uniqueness}]
Let $f_1,f_2 \in \mathcal{X}$ be two weak solutions of \eqref{eq: model 4 zero pe} with identical admissible initial data, and define $\rho_1,\rho_2$ accordingly. As $\rho_1-\rho_2$ satisfies the heat equation with zero initial data, the standard theory implies $\rho_1=\rho_2$; we therefore denote this common function by $\rho$, which satisfies the usual estimate $0\leq \rho \leq 1$. It therefore follows that
\begin{equation}
    \partial_t f_i = D_e\dv((1-\rho)\nabla f_i + f_i \nabla \rho) + \partial^2_\theta f_i 
\end{equation}
for $i \in \{1,2\}$. 

Consider the distance 
\begin{equation}\label{eq:gajewski distance}
    d(u,v) := \int_\Upsilon \Big( \zeta(u) + \zeta(v) -2\zeta\big(\frac{u+v}{2}\big) \Big) \d \bxi, 
\end{equation}
where $\zeta(s) = s(\log s -1) + 1$ for $s \geq 0$. Note that the convexity of $\zeta$ implies that $d(u,v) \geq 0$. In fact we shall consider the regularised variant for $\delta > 0$: 
\begin{equation}\label{eq:gajewski distance reg}
    d_\delta(u,v) := \int_\Upsilon \Big( \zeta_\delta(u) + \zeta_\delta(v) -2\zeta_\delta\big(\frac{u+v}{2}\big) \Big) \d \bxi, 
\end{equation}
where $\zeta_\delta(s) = \zeta(s+\delta)$, and again $d_\delta$ is non-negative because $\zeta_\delta$ is convex. Note that $\zeta'(s) = \log(s)$ for $s>0$ and that, by a standard argument using Taylor's Theorem, there holds for all $\delta > 0$: 
\begin{equation}\label{eq:est convex comb for gajewski metric}
    \zeta_\delta(u) + \zeta_\delta(v) -2\zeta_\delta\big(\frac{u+v}{2}\big) \geq \frac{1}{8}|u-v|^2; 
\end{equation}
we refer the reader to \cite[\S 6]{JungelZamponi} for further details.

Our goal is to show that $f_1 \equiv f_2$ a.e.~in $\Upsilon_T$. With this in mind, we compute 
\begin{equation}\label{eq:uniqueness expanded equals}
    \begin{aligned}
        \frac{\de}{\de t}d_\delta(f_1,f_2) = &\langle \partial_t f_1, \log(f_1+\delta) \rangle + \langle \partial_t f_2 , \log(f_2+\delta) \rangle  - \langle \partial_t f_1 + \partial_t f_2 , \log\big(\frac{f_1+f_2}{2}+\delta  \big) \rangle  \\ 
        =& - 4 D_e \int_\Upsilon \Big( \sum_{i=1}^2 |\nabla \sqrt{f_i+\delta}|^2 - |\nabla\sqrt{f_1+f_2+2\delta}|^2 \Big) (1-\rho) \d \bxi \\ 
        &- 4  \int_\Upsilon \Big(  \sum_{i=1}^2|\partial_\theta \sqrt{f_i+\delta}|^2 -  |\partial_\theta \sqrt{f_1+f_2+2\delta}|^2\Big) \d \bxi \\ 
        &- D_e \sum_{i=1}^2 \int_\Upsilon \Big(\frac{f_i}{f_1+\delta} - \frac{f_1+f_2}{f_1+f_2+2\delta}  \Big)\nabla f_i \cdot \nabla \rho \d \bxi. 
    \end{aligned}
\end{equation}
The first two integrals above are non-negative by virtue of the subadditivity of the Fisher information \cite[Lemma 9]{JungelZamponi}; \textit{cf.}~\cite[\S 3.6]{Pacard}. Our strategy will be to show that the final term on the right-hand side of the previous equation vanishes in the limit as $\delta \to 0$. It will suffice to first show that the term $\nabla f_i \cdot \nabla \rho$ is integrable, and to then apply the Dominated Convergence Theorem. We rewrite this product as 
\begin{equation*}
    \nabla f_i \cdot \nabla \rho = -4 \sqrt{1-\rho} \nabla \sqrt{f_i} \cdot \sqrt{f_i}\nabla\sqrt{1-\rho}, 
\end{equation*}
from which it follows, using the Young inequality, that 
\begin{equation*}
   \begin{aligned}
       \int_{\Upsilon_T} |\nabla f_i \cdot \nabla \rho| &\leq 2\int_{\Upsilon_T} (1-\rho)|\nabla_{\bxi} \sqrt{f_i}|^2 \d \bxi \d t + 2 \int_0^T \Big( \int_\Omega |\nabla \sqrt{1-\rho}|^2 \underbrace{\int_0^{2\pi} f_i \d \theta}_{=\rho} \Big) \d x \d t \\ 
       &\leq 2\int_{\Upsilon_T} (1-\rho)|\nabla_{\bxi} \sqrt{f_i}|^2 \d \bxi \d t + 2 \int_{\Omega_T} |\nabla \sqrt{1-\rho}|^2 \d x \d t, 
   \end{aligned} 
\end{equation*}
and thus $\nabla f_i \cdot \nabla \rho \in L^1(\Upsilon_T)$; the boundedness of the right-hand side follows from the definition of $\mathcal{X}$. It follows that, by integrating \eqref{eq:uniqueness expanded equals} in time, there holds, for a.e.$t \in (0,T)$, 
\begin{equation*}
    d_\delta(f_1,f_2)(t) - \underbrace{d_\delta(f_1,f_2)(0)}_{=0} \leq D_e \sum_{i=1}^2 \int_0^t \int_{\Upsilon} \Big(\frac{f_i}{f_1+\delta} - \frac{f_1+f_2}{f_1+f_2+2\delta}  \Big)\nabla f_i \cdot \nabla \rho \d \bxi \d \tau, 
\end{equation*}
whence, by the Dominated Convergence Theorem and using $d_\delta \geq 0$, we deduce that, for a.e.~$t \in (0,T)$, 
\begin{equation*}
    \lim_{\delta \to 0^+}d_\delta(f_1,f_2)(t) = 0. 
\end{equation*}
It follows from Fatou's Lemma that 
\begin{equation*}
    \lim_{\delta \to 0^+} \Big( \zeta_\delta(f_1) + \zeta_\delta(f_1) -2\zeta_\delta\big(\frac{f_1+f_2}{2}\big) \Big) = 0 \quad \text{a.e.~in } \Upsilon_T, 
\end{equation*}
whence the estimate \eqref{eq:est convex comb for gajewski metric} implies that $f_1 \equiv f_2$ a.e.~in $\Upsilon_T$, as required. 
\end{proof}

\section{Existence and Uniqueness of Stationary States}\label{sec:stat states}

We consider the existence and uniqueness of stationary solutions $f_\infty : \Upsilon \to \mathbb{R}$ which satisfy, with $\rho_\infty(x) = \int_0^{2\pi} f_\infty(x,\theta) \d \theta$, 
\begin{equation}\label{eq:stat state}
    \pe \dv\big((1-\rho_\infty) f_\infty \e(\theta)\big) = D_e \dv\big( (1-\rho_\infty)\nabla f_\infty + f_\infty \nabla \rho_\infty \big) + \partial^2_\theta f_\infty. 
\end{equation}
The question of existence of stationary states is straightforward, with no restriction on the P\'eclet number. Indeed, given any $\pe \in \mathbb{R}$, there exists at least one non-trivial constant stationary solution; namely $f_\infty = 1/2\pi$, which gives rise to $\rho_\infty =1$. In fact, by writing the drift term as 
\begin{equation*}
    \dv\big((1-\rho_\infty) f_\infty \e(\theta)\big) =  \e(\theta) 
 \cdot \nabla\big( (1-\rho_\infty) f_\infty \big), 
\end{equation*}
for sufficiently regular $f_\infty$, we deduce that all constant solutions satisfy \eqref{eq:stat state}.

We therefore focus entirely on uniqueness henceforth, and divide our results into two subsections; \S \ref{sec:stat null peclet} focused exclusively on the case $\pe=0$ by means of a simple variational approach, and \S \ref{sec:stat fixed point} concerned with uniqueness of strong solutions near constant stationary states for general $\pe$, which requires a more involved analysis.

\subsection{Uniqueness of stationary solutions for null P\'eclet number}\label{sec:stat null peclet}

The case $\pe=0$ is particularly simple since, in this case, the equation \eqref{eq: model 4} admits an exact gradient flow structure. We shall show that stationary states are unique up to a mass constraint via a variational approach. The proof does not readily generalise to $\pe \neq 0$; this is essentially because the perturbative term $V$ in \eqref{eq:grad flow perturbation in terms of u} does not admit a primitive, whence $\nabla_{\bxi}\cdot(\tilde{M}(u)(\nabla_{\bxi}u+V))$ cannot be rewritten as the dissipation of a suitable energy (\textit{cf.}~\cite[\S 3.3]{BoundEntropy}).

\begin{thm}\label{thm:euler lagrange stat states}
 Let $\pe=0$ and $m>0$. Then, the stationary solution $f_\infty$ of \eqref{eq:stat state} is unique within the class $\int_\Upsilon f_\infty \d \bxi = m$. 
\end{thm}

Further results regarding the convergence to stationary states can be obtained using the theory of Bakry--Emery or log-Sobolev inequalities applied to equations endowed with an exact Wasserstein gradient flow structure (\textit{cf.}~\textit{e.g.}~\cite{AGS08,Carrillo:2003uk,JungelEntMeth}). This is well-known, and we do not include the details here. {The study of the convergence to stationary states by means of hypocoercivity methods for the case of non-zero P\'eclet number, for which one does not have an exact gradient flow structure, will be the subject of future investigation.}

\begin{proof}

\noindent 1. \textit{Variational consequence of stationary equation}: As per \S \ref{sec:calc var}, we introduce the transformed variable $u_\infty$ 
\begin{equation}\label{eq:u infty}
    u_\infty := \log \big( \frac{f_\infty}{1-\rho_\infty} \big). 
\end{equation}
For clarity of presentation, we assume that the above is well-defined and perform formal computations; these may easily be rendered rigorous by standard approximation arguments. We show that any stationary solution satisfies 
    \begin{equation}\label{eq:euler lagrange type}
        \int_\Upsilon f_\infty (1-\rho_\infty) |\nabla_{\bxi} u_\infty|^2 \d \bxi = 0. 
    \end{equation}
Indeed, define the functional 
    \begin{equation*}
        F[f] := \int_\Upsilon \Big( D_e (1-\rho)f|\nabla u|^2 + f|\partial_\theta u |^2 \Big) \d \bxi = \int_\Upsilon \tilde{M}(u)\nabla_{\bxi} u \cdot \nabla_{\bxi} u \d \bxi \geq 0, 
    \end{equation*}
  where we used that $\tilde{M}(u)$ is positive semidefinite, and observe that testing the equation \eqref{eq:stat state} with $u_\infty$ yields $F[f_\infty]=0$. The equality \eqref{eq:euler lagrange type} follows from the non-negativity of the integrand in the functional $F$ and basic manipulations. 
   
   \noindent 2. \textit{Uniqueness from mass constraint}: The relation \eqref{eq:euler lagrange type} implies that all stationary solutions with $\pe=0$ are constant. The mass constraint $\int_\Upsilon f_\infty \d \bxi = m$ then uniquely determines the stationary solution, as required. 
\end{proof}

\subsection{Uniqueness of strong solutions near constant stationary states}\label{sec:stat fixed point}

Our aim in this section is to prove that, provided a strong solution is sufficiently near a constant stationary state (in a sense to be made precise), then it is the only such solution. We proceed by means of a fixed-point argument employing the Contraction Mapping Theorem, similar to \cite[\S 3.5]{BoundEntropy}, for which we define the spaces 
\begin{equation}\label{eq:space for fixed point}
  \begin{aligned}
      &\Xi_t := L^\infty(t,T;H^2(\Upsilon)) \cap L^2(t,T;H^3(\Upsilon)) \cap H^1(t,T;H^1(\Upsilon)), 
      \\ 
      &\Lambda_t := L^\infty(t,T;H^1(\Upsilon)) \cap L^2(t,T;H^2(\Upsilon)) \cap H^1(t,T;L^2(\Upsilon)), \\ 
      &\Theta_t := L^\infty(t,T;L^2(\Upsilon)) \cap L^2(t,T;H^1(\Upsilon)). 
  \end{aligned} 
\end{equation}
Where no confusion arises, we shall omit the $t$ subscript and simply write $\Xi,\Lambda,\Theta$. The main result of this section is as follows. 

\begin{thm}\label{thm:uniqueness small peclet}

Let $D_e > 0$, $\pe \in \mathbb{R}$, $f_\infty \in (0,1/2\pi)$ be a given constant stationary state, and $f$ be a solution of \eqref{eq: model 4}. Assume there exists $t \in (0,T)$ such that $f(t,\cdot) \in H^2(\Upsilon)$ with 
\begin{equation*}
    \Vert f(t,\cdot) - f_\infty \Vert_{H^2(\Upsilon)} \leq R^2 
\end{equation*}
for $R$ sufficiently small, depending only on $f_\infty,D_e$ and $\pe$. Then, there exists a unique solution of \eqref{eq: model 4} in 
\begin{equation*}
    B_{R,t} := \big\lbrace f \in \Xi : \, \Vert f - f_\infty \Vert_{\Xi_t} \leq R \big\rbrace. 
\end{equation*}
\end{thm}

The proof of Theorem \ref{thm:uniqueness small peclet} entails studying the time-evolution of the error term 
$w := f - f_\infty$, with $W := \rho - \rho_\infty$, which satisfies 
\begin{equation*}
    \begin{aligned}
        \partial_t w - \partial^2_\theta w -D_e \big[ (1-\rho_\infty)\Delta w & + f_\infty \Delta W \big] +\pe \big[(1-\rho_\infty) \nabla w  - f_\infty \nabla W \big]\cdot \e(\theta)  \\ 
        =& D_e \big[ w \Delta W - W\Delta w \big] + \pe \big[w \nabla W + W \nabla w \big] \cdot \e(\theta). 
    \end{aligned}
\end{equation*}
We therefore construct the operator $\Gamma$ to be the solution map $\Gamma : w \mapsto z$, where $z$ solves the linear problem 
\begin{equation}\label{eq:define fixed point}
   \left\lbrace \begin{aligned}
       & \partial_t z - \partial^2_\theta z -D_e \big[ (1-\rho_\infty)&&\Delta z + f_\infty  \Delta Z \big]  +\pe \big[(1-\rho_\infty) \nabla z  - f_\infty \nabla Z \big]\cdot \e(\theta)\\ 
     &   &&= \underbrace{D_e \big( w \Delta W - W\Delta w \big) + \pe \big( w \nabla W + W \nabla w \big) \cdot \e(\theta)}_{=:G(w)}, \\ 
        &z(0,\cdot) = z_0, &&
    \end{aligned}\right. 
\end{equation}
and $Z := \int_0^{2\pi} z \d \theta$; we write $\Gamma = S \circ G$, where $S$ is the solution mapping to the above equation for given right-hand side and initial data $z_0$, and $G$ maps precisely to this desired right-hand source term. Our goal is to show that $\Gamma$ admits a fixed point. 

Note that, provided the above is supplemented with adequate initial data and the right-hand term $G(w)$ is sufficiently regular, the solution map $\Gamma$ is well-defined in a suitable space using a standard application of the Galerkin method/Fourier series. Indeed, writing the ansatz 
\begin{equation*}
    z(t,\bxi) = \sum_{\n \in \mathbb{Z}^3} \alpha_\n(t) e^{\i \n \cdot \bxi}, \quad Z(t,x) = 2\pi \sum_{\n' \in \mathbb{Z}^2} \alpha_{(\n',0)}(t) e^{\i \n' \cdot x}, 
\end{equation*}
where $\n = (n_1,n_2,n_3)$, $\n' = (n_1,n_2)$, the equation \eqref{eq:define fixed point} reads 
\begin{equation}\label{eq:ode fourier for fixed point setup}
    \begin{aligned}
        \alpha_\n'(t) + n_3^2 \alpha_\n(t) + D_e (n_1^2+n_2^2)\big[ (1-\rho_\infty) \alpha_\n(t) &+ \mathds{1}_{\{n_3 = 0\}}f_\infty 2\pi \alpha_{(\n',0)}(t) \big] \\ 
        &+ \pe (1-\rho_\infty) 2 \pi \i n_1 \alpha_\n(t) = \hat{G}_\n(t), 
    \end{aligned}
\end{equation}
where 
\begin{equation*}
    \hat{G}_\n (t) = \frac{1}{(2\pi)^2}\int_0^{2\pi} G(w)(t,\bxi) e^{-\i \n \cdot \bxi} \d \bxi. 
\end{equation*}
It follows that, provided $\hat{G}_\n(w)$ is sufficiently well-behaved, the ODE system \eqref{eq:ode fourier for fixed point setup} is solvable by standard methods for all values of $f_\infty \geq 0$ and $0 \leq \rho_\infty \leq 1$. This existence construction is standard; we therefore deliberately omit further details for clarity of presentation, and henceforth focus exclusively on the \emph{a priori} estimates required for the fixed-point strategy. Our results are encapsulated in the following two lemmas, which will be used to prove that $\Gamma$ admits a unique fixed point.

\begin{lemma}[Estimates for the operator $S$]\label{eq:parabolic estimates for fixed pt}
    Assume $f_\infty \in (0,1/2\pi)$, \textit{i.e.}, $\rho_\infty \in (0,1)$. Then, there exists a positive constant $C_{\infty,D_e} = C_{\infty,D_e}(f_\infty,D_e)$ such that 
    \begin{equation}\label{eq:first est stat states fixed point}
        \begin{aligned}
            &\Vert S v \Vert_{\Theta} \leq C_{\infty,D_e} C_0^{1/2} \Big( \Vert z_0 \Vert_{L^2(\Upsilon)} + \Vert v \Vert_{L^2(\Upsilon_T)} \Big), \\ 
            &\Vert \nabla S v \Vert_\Theta \leq C_{\infty,D_e} C_0 \Big( \Vert \nabla z_0 \Vert_{L^2(\Upsilon)} + \Vert v \Vert_{L^2(\Upsilon_T)} \Big), \\ 
            &\Vert \partial_\theta S v \Vert_\Theta \leq C_{\infty,D_e} C_0^{3/2} \Big( \Vert z_0 \Vert_{L^2(\Upsilon)} + \Vert \partial_\theta z_0 \Vert_{L^2(\Upsilon)} + \Vert v \Vert_{L^2(\Upsilon_T)} \Big), \\ 
             &\Vert \partial_t Sv \Vert_{L^2(\Upsilon_T)} \leq  C_{\infty,D_e} C_1 \Big( \Vert z_0 \Vert_{H^1(\Upsilon)} + \Vert v \Vert_{L^2(\Upsilon_T)} \Big), 
        \end{aligned}
    \end{equation}

    where 
    \begin{equation}\label{eq:fp c0}
  \begin{aligned}
      C_0 = e^{T(1+\pe^2{D_e^{-1}})}, \qquad C_1 = (1+\pe^2 D_e^{-1})^{1/2} C_0^{3/2}, 
  \end{aligned}
\end{equation} 
and 
\begin{equation}\label{eq:main fixed pt est for S}
    \Vert Sv \Vert_\Xi \leq C_{\infty,D_e} C_1 \Big( \Vert z_0 \Vert_{H^2(\Upsilon)} + \Vert v \Vert_{\Theta} \Big). 
\end{equation}
\end{lemma}

\begin{lemma}[Estimates for the operator $G$]\label{eq:parabolic estimates for fixed pt ii}
    There exists a positive constant $C_{D_e}$, depending solely on $D_e$ and $\Upsilon$, such that, for all $w_1,w_2 \in \Xi$, there holds 
    \begin{equation}\label{eq:G contractive}
        \Vert G(w_1) - G(w_2) \Vert_\Theta \leq C_{D_e}(1+\pe D_e^{-1}) \Vert w_1 - w_2 \Vert_\Xi \big( \Vert w_1 \Vert_\Xi + \Vert w_2 \Vert_{\Xi} \big), 
    \end{equation}
    and, for all $w \in \Xi$, 
    \begin{equation}\label{eq:self mapping for G}
        \Vert G(w) \Vert_\Theta \leq C_{D_e}(1+\pe D_e^{-1}) \Vert w \Vert_\Xi^2. 
    \end{equation}
\end{lemma}

Before proving Lemmas \ref{eq:parabolic estimates for fixed pt} and \ref{eq:parabolic estimates for fixed pt ii}, we provide the proof of Theorem \ref{thm:uniqueness small peclet}. 

\begin{proof}[Proof of Theorem \ref{thm:uniqueness small peclet}]

Suppose for the time being that $T \leq 1$. This assumption will be removed in the third step of the proof. 

\noindent 1. \textit{$\Gamma$ is self-mapping}: Suppose $f \in B_R$. Using the estimates of Lemmas \ref{eq:parabolic estimates for fixed pt} and \ref{eq:parabolic estimates for fixed pt ii}, we find that 
\begin{equation*}
    \Vert \Gamma w \Vert_\Xi \leq C_{\infty,D_e} C_1 \Big( 1 + C_{D_e}(1+\pe D_e^{-1}) \Big)R^2. 
\end{equation*}
Since we have restricted $T \leq 1$, it follows that there exists $R$ sufficiently small, depending only on $D_e,f_\infty,\pe$ such that there holds 
\begin{equation*}
    \Vert \Gamma w \Vert_\Xi \leq R, 
\end{equation*}
whence $\Gamma$ is self-mapping from $B_R$ to itself. 

\noindent 2. \textit{$\Gamma$ is contractive}: Let $w_1,w_2 \in B_R$. Observe that the linearity of $S$ implies $\Gamma w_1 - \Gamma w_2 = S(Gw_1 - Gw_2)$, where the initial data $z_0$ in \eqref{eq:define fixed point} is null. It follows from Lemmas \ref{eq:parabolic estimates for fixed pt} and \ref{eq:parabolic estimates for fixed pt ii} that there holds 
\begin{equation*}
    \Vert \Gamma w_1 - \Gamma w_2 \Vert_\Xi \leq C_{\infty,D_e} C_1 C_{D_e}(1+\pe D_e^{-1})R \Vert w_1 - w_2 \Vert_\Xi. 
\end{equation*}
As per the first part of the proof, it follows that one may select $R$ sufficiently small, depending only on $f_\infty,D_e,\pe$ such that $\Gamma$ is a contraction. An application of Banach's Contraction Mapping Theorem implies that there exists a unique $f$ solving \eqref{eq: model 4} and such that $\Vert f - f_\infty \Vert_\Xi$ over the time interval $t \in (0,1)$. 

\noindent 3. \textit{Extension to arbitrary time intervals}: We apply Step 2 of the proof repeatedly over overlapping time intervals $(\alpha_j,\beta_j)$ for Lebesgue points $\alpha_j,\beta_j$ chosen such that $\beta_{j-1}>\alpha_j$ for all $j$ and such that $\bigcup_j (\alpha_j,\beta_j) \supset (t,T)$. In turn, obtain that there exists a unique $f \in B_R$ over the entire time interval $(0,T)$. 
\end{proof}

\begin{proof}[Proof of Lemma \ref{eq:parabolic estimates for fixed pt}]
    By testing \eqref{eq:define fixed point} with $z$, we obtain the classical parabolic estimate 
\begin{equation*}
    \begin{aligned}
        \frac{1}{2}\frac{\de}{\de t}\int_\Upsilon |z|^2 \d \bxi + \int_\Upsilon & |\partial_\theta z|^2 \d \bxi + D_e \Big( (1-\rho_\infty) \int_\Upsilon |\nabla z |^2 \d \bxi + f_\infty \int_\Omega |\nabla Z|^2 \d x \Big) \\ 
        \leq & \pe \int_\Upsilon \big[  (1-\rho_\infty) |z| |\nabla z| +  f_\infty | z| |\nabla Z| \big] \d \bxi + \int_\Upsilon G(w) z \d \bxi \\ 
        \leq& \frac{1}{2}D_e(1-\rho_\infty) \int_\Upsilon |\nabla z |^2 \d \bxi + \frac{1}{2}D_e f_\infty\int_\Omega |\nabla Z|^2 \d x \\ 
        &+ \frac{1}{2}\big((1-\rho_\infty+ 2\pi f_\infty)\pe^2{D_e^{-1}} + 1\big)\int_\Upsilon |z|^2 \d \bxi + \frac{1}{2}\int_\Upsilon |G(w)|^2 \d \bxi 
    \end{aligned}
\end{equation*}
for some positive universal constant $C$, from which we get, for a.e.~$t \in (0,T)$, 
\begin{equation*}
    \begin{aligned}
        \int_\Upsilon |z(t)|^2 \d \bxi +& \int_0^t\int_\Upsilon |\partial_\theta z|^2 \d \bxi \d \tau + D_e \Big( (1-\rho_\infty) \int_0^t \int_\Upsilon |\nabla z |^2 \d \bxi \d \tau + f_\infty \int_0^t \int_\Omega |\nabla Z|^2 \d x \d \tau \Big) \\ 
        \leq & (1+\pe^2{D_e^{-1}})\int_0^t\int_\Upsilon |z|^2 \d \bxi \d \tau + \Big(\int_\Upsilon |z_0|^2 \d \bxi + \int_0^t 
 \int_{\Upsilon} |G(w)|^2 \d \bxi \d \tau \Big), 
    \end{aligned}
\end{equation*}
and Gr\"onwall's Lemma yields 
\begin{equation*}
    \begin{aligned}
        \int_\Upsilon |z(t)|^2 \d \bxi 
        \leq & \Big(\int_\Upsilon |z_0|^2 \d \bxi + \int_0^t 
 \int_{\Upsilon} |G(w)|^2 \d \bxi \d \tau \Big)e^{t(1+\pe^2{D_e^{-1}})}, 
    \end{aligned}
\end{equation*}
so that 
\begin{equation*}
    \begin{aligned}
       \int_0^T \int_\Upsilon |z(t)|^2 \d \bxi \d t 
        \leq & \Big(\Vert z_0 \Vert_{L^2(\Upsilon)}^2 + \Vert G(w) \Vert_{L^2(\Upsilon_T)}^2  \Big) \frac{1}{(1+\pe^2{D_e^{-1}})}\big(e^{T(1+\pe^2{D_e^{-1}})}-1 \big), 
    \end{aligned}
\end{equation*}
whence 
\begin{equation}\label{eq:to get first fp est}
    \begin{aligned}
       \Vert z \Vert_{L^\infty(0,T;L^2(\Upsilon))}^2 + \Vert \partial_\theta z \Vert_{L^2(\Upsilon_T)}^2 + D_e \Big( (1-\rho_\infty) \Vert & \nabla z \Vert_{L^2(\Upsilon_T)}^2 + f_\infty \Vert \nabla Z \Vert_{L^2(\Omega_T)}^2 \Big) \\ 
        \leq & C_0 \Big(\Vert z_0 \Vert_{L^2(\Upsilon)}^2 + \Vert G(w) \Vert_{L^2(\Upsilon_T)}^2 \Big), 
    \end{aligned}
\end{equation}
and the first estimate in \eqref{eq:first est stat states fixed point} follows, where $C_0 = C_0(\pe,D_e,T)$ is independent of $f_\infty,z_0$ and is given by \eqref{eq:fp c0}.

Similarly, by testing \eqref{eq:define fixed point} with $(1-\rho_\infty)\Delta z + f_\infty \Delta Z$, we get 
\begin{equation*}
    \begin{aligned}
     \frac{1}{2} \frac{\de}{\de t}\Big((1-\rho_\infty) \int_\Upsilon & |\nabla z|^2 \d \bxi + f_\infty \int_\Omega |\nabla Z|^2 \d x  \Big) \\ 
      +& \frac{1}{2}D_e \int_\Upsilon \big| (1-\rho_\infty)\Delta z + f_\infty \Delta Z \big|^2 \d \bxi  
      + (1-\rho_\infty) \int_\Upsilon |\nabla \partial_\theta z|^2 \d \bxi \\ 
      &\leq \frac{1}{2}\pe^2 D_e^{-1} \int_\Upsilon |(1-\rho_\infty) \nabla z - f_\infty \nabla Z|^2 \d \bxi + \frac{1}{2} D_e^{-1}\int_\Upsilon |G(w)|^2 \d \bxi. 
    \end{aligned}
\end{equation*}
It follows that 
\begin{equation*}
    \begin{aligned}
      \Big((1-&\rho_\infty) \int_\Upsilon |\nabla z|^2 \d \bxi + f_\infty \int_\Omega |\nabla Z|^2 \d x  \Big)(t) \\ 
      &+ D_e \int_0^t \int_\Upsilon \big| (1-\rho_\infty)\Delta z + f_\infty \Delta Z \big|^2 \d \bxi \d \tau 
      + (1-\rho_\infty) \int_0^t \int_\Upsilon |\nabla \partial_\theta z|^2 \d \bxi \d \tau \\ 
      \leq& \Big((1-\rho_\infty) \int_\Upsilon |\nabla z_0|^2 \d \bxi + f_\infty \int_\Omega |\nabla Z_0|^2 \d x  \Big) + D_e^{-1} \int_0^t \int_\Upsilon |G(w)|^2 \d \bxi \d \tau \\ 
      &+\pe^2 D_e^{-1} \Big((1-\rho_\infty)^2 \int_0^t \int_\Upsilon |\nabla z|^2 \d \bxi \d \tau + (2(1-\rho_\infty)f_\infty + 2\pi f_\infty^2) \int_0^t \int_\Omega |\nabla Z|^2 \d x \d \tau \Big), 
    \end{aligned}
\end{equation*}
whence, since $2\pi f_\infty = \rho_\infty \in [0,1]$, we get 
\begin{equation*}
    \begin{aligned}
      \Big((1-\rho_\infty) & \int_\Upsilon |\nabla z|^2 \d \bxi + f_\infty \int_\Omega |\nabla Z|^2 \d x  \Big)(t) \\ 
      &+ D_e \int_0^t \int_\Upsilon \big| (1-\rho_\infty)\Delta z + f_\infty \Delta Z \big|^2 \d \bxi \d \tau 
      + (1-\rho_\infty) \int_0^t \int_\Upsilon |\nabla \partial_\theta z|^2 \d \bxi \d \tau \\ 
      \leq& \Big((1-\rho_\infty) \int_\Upsilon |\nabla z_0|^2 \d \bxi + f_\infty \int_\Omega |\nabla Z_0|^2 \d x  \Big) + D_e^{-1} \int_0^t \int_\Upsilon |G(w)|^2 \d \bxi \d \tau \\ 
      &+2\pe^2 D_e^{-1} \Big((1-\rho_\infty) \int_0^t \int_\Upsilon |\nabla z|^2 \d \bxi \d \tau + f_\infty \int_0^t \int_\Omega |\nabla Z|^2 \d x \d \tau \Big). 
    \end{aligned}
\end{equation*}
Gr\"onwall's Lemma yields, for a.e.~$t \in (0,T)$, 
\begin{equation*}
    \begin{aligned}
         \Big(&(1-\rho_\infty) \int_\Upsilon |\nabla z|^2 \d \bxi + f_\infty \int_\Omega |\nabla Z|^2 \d x  \Big)(t) \\ 
         &\leq \Big[  (1-\rho_\infty) \Vert \nabla z_0 \Vert_{L^2(\Upsilon)}^2 + f_\infty \Vert \nabla Z_0\Vert_{L^2(\Omega)}^2 + D_e^{-1} \Vert G(w)\Vert_{L^2((0,t)\times\Upsilon)}^2 \Big] e^{t 2\pe^2 D_e^{-1}}. 
    \end{aligned}
\end{equation*}
We therefore obtain 
\begin{equation*}
    \begin{aligned}
     (1-\rho_\infty) \Vert \nabla z \Vert_{L^\infty(0,T;L^2(\Upsilon))}^2 & + f_\infty \Vert \nabla Z\Vert_{L^\infty(0,T;L^2(\Omega))}^2 \\ 
      + D_e \Vert (1-\rho_\infty)\Delta z & + f_\infty \Delta Z \Vert_{L^2(\Upsilon_T)}^2
      + (1-\rho_\infty) \Vert \nabla \partial_\theta z\Vert^2_{L^2(\Upsilon_T)} \\ 
      \leq C_0^2 \Big( (1-\rho_\infty) & \Vert \nabla z_0 \Vert_{L^2(\Upsilon)}^2 + f_\infty \Vert \nabla Z_0\Vert_{L^2(\Omega)}^2 + D_e^{-1}\Vert G(w)\Vert_{L^2(\Upsilon_T)}^2 \Big). 
    \end{aligned}
\end{equation*}
Note in particular that the Fubini--Tonelli Theorem implies 
\begin{equation*}
    \begin{aligned}
        \Vert (1-\rho_\infty)\Delta z + f_\infty \Delta Z \Vert_{L^2(\Upsilon_T)}^2 = (1-\rho_\infty)^2 \Vert \Delta z \Vert^2_{L^2(\Upsilon_T)} + f_\infty (2(1-\rho_\infty) + 2\pi f_\infty) \Vert \Delta Z \Vert^2_{L^2(\Omega_T)}, 
    \end{aligned}
\end{equation*}
whence the previous estimate yields 
\begin{equation}\label{eq:second bit for fp}
    \begin{aligned}
     (1-\rho_\infty) \Vert \nabla z \Vert_{L^\infty(0,T;L^2(\Upsilon))}^2 & + f_\infty \Vert \nabla Z\Vert_{L^\infty(0,T;L^2(\Omega))}^2  + (1-\rho_\infty) \Vert \nabla \partial_\theta z\Vert^2_{L^2(\Upsilon_T)} \\ 
      + D_e \Big( (1-\rho_\infty)^2& \Vert \Delta z \Vert^2_{L^2(\Upsilon_T)} + f_\infty (2-\rho_\infty) \Vert \Delta Z \Vert^2_{L^2(\Omega_T)} \Big) \\ 
      \leq C_0^2 \Big( & (1-\rho_\infty) \Vert \nabla z_0 \Vert_{L^2(\Upsilon)}^2 + f_\infty \Vert \nabla Z_0\Vert_{L^2(\Omega)}^2 + D_e^{-1}\Vert G(w)\Vert_{L^2(\Upsilon_T)}^2 \Big), 
    \end{aligned}
\end{equation}
and the second estimate in \eqref{eq:first est stat states fixed point} follows. Similarly, we test \eqref{eq:define fixed point} against $\partial^2_\theta z$ to obtain 
\begin{equation*}
    \begin{aligned}
        \frac{1}{2}\frac{\de}{\de t}\int_{\Upsilon} |\partial_\theta z|^2 \d \bxi + & \int_{\Upsilon} |\partial^2_\theta z| \d \bxi + D_e(1-\rho_\infty)\int_{\Upsilon} |\nabla \partial_\theta z |^2 \d \bxi \\ 
        = &\pe  \int_{\Upsilon} \big[ (1-\rho_\infty)\nabla z - f_\infty \nabla Z] \cdot \e(\theta) \partial^2_\theta z \d \bxi - \int_{\Upsilon} G(w) \partial^2_\theta z \d \bxi \\ 
        =&-\pe (1-\rho_\infty) \int_{\Upsilon} \nabla \partial_\theta z \cdot \e(\theta) \partial_\theta z \d \bxi \\ 
        &-\pe \int_{\Upsilon} \big[ (1-\rho_\infty)\nabla z - f_\infty \nabla Z \big] \cdot \e'(\theta) \partial_\theta z \d \bxi - \int_\Upsilon G(w) \partial^2_\theta z \d \bxi.  
    \end{aligned}
\end{equation*}
Using $\e'(\theta) = (-\sin\theta,\cos\theta)$ and the Young inequality, we get 
\begin{equation*}
        \begin{aligned}
        \frac{1}{2}\frac{\de}{\de t}\int_{\Upsilon} |\partial_\theta z|^2 \d \bxi &+ \int_{\Upsilon} |\partial^2_\theta z| \d \bxi + D_e(1-\rho_\infty)\int_{\Upsilon} |\nabla \partial_\theta z |^2 \d \bxi \\ 
        \leq&\pe (1-\rho_\infty) \int_{\Upsilon} |\nabla \partial_\theta z| |\partial_\theta z |\d \bxi + \pe (1-\rho_\infty) \int_{\Upsilon} |\nabla z| |\partial_\theta z| \d \bxi \\ 
        &+\pe f_\infty \int_{\Upsilon} |\nabla Z| |\partial_\theta z| \d \bxi + \int_\Upsilon |G(w)| |\partial^2_\theta z| \d \bxi.  
    \end{aligned}
\end{equation*}
We obtain, for a.e.~$t \in (0,T)$, 
\begin{equation*}
    \begin{aligned}
        \Vert\partial_\theta z\Vert_{L^\infty(0,t;L^2(\Upsilon))}^2 +\Vert \partial^2_\theta z &\Vert^2_{L^2(\Upsilon_t)} + D_e(1-\rho_\infty)\int_{\Upsilon_t} |\nabla \partial_\theta z |^2 \d \bxi \\ 
        \leq &2\pe^2 D_e^{-1} (1-\rho_\infty) \Vert\partial_\theta z\Vert^2_{L^2(\Upsilon_t)} + D_e (1-\rho_\infty) \Vert\nabla z \Vert^2_{L^2(\Upsilon_T)} \\ 
        &+ \Vert \partial_\theta z_0 \Vert^2_{L^2(\Upsilon)} + \Vert G(w) \Vert^2_{L^2(\Upsilon_t)} \\ 
       \leq& 2\pe^2 D_e^{-1} \Vert\partial_\theta z\Vert^2_{L^2(\Upsilon_t)} \\ 
       &+ C_0 \Big(\Vert z_0 \Vert_{L^2(\Upsilon)}^2 + \Vert \partial_\theta z_0 \Vert^2_{L^2(\Upsilon)} + 2\Vert G(w) \Vert^2_{L^2(\Upsilon_t)}\Big), 
    \end{aligned}
\end{equation*}
where we used $\rho_\infty \in [0,1]$ and estimate \eqref{eq:to get first fp est} to obtain the final line. Gr\"onwall's Lemma yields 
\begin{equation}\label{eq:third bit for fp}
    \begin{aligned}
       \Vert\partial_\theta z\Vert_{L^\infty(0,T;L^2(\Upsilon))}^2 + \Vert \partial^2_\theta z \Vert^2_{L^2(\Upsilon_T)} + & D_e(1-\rho_\infty) \Vert \nabla \partial_\theta z \Vert_{L^2(\Upsilon_T)}^2 \\ 
        \leq & C_0^3 \Big( \Vert z_0 \Vert_{L^2(\Upsilon)}^2 + \Vert \partial_\theta z_0 \Vert^2_{L^2(\Upsilon)} + 2\Vert G(w) \Vert_{L^2(\Upsilon_T)}^2 \Big), 
    \end{aligned}
\end{equation}
and the third estimate in \eqref{eq:first est stat states fixed point} follows. It follows, using directly the equation \eqref{eq:define fixed point}, that there holds, for some positive universal constant $C$, the estimate on the time derivative: 
\begin{equation*}
    \begin{aligned}
        \Vert \partial_t z \Vert_{L^2(\Upsilon_T)} \leq C\Big(& \Vert \partial^2_\theta z \Vert_{L^2(\Upsilon_T)}^2 + D_e^2 (1-\rho_\infty)^2 \Vert \Delta z \Vert_{L^2(\Upsilon_T)}^2 + D_e^2 f_\infty^2 \Vert \Delta Z \Vert_{L^2(\Upsilon_T)}^2 \\ 
        &+ \pe^2 (1-\rho_\infty)^2 \Vert \nabla z \Vert_{L^2(\Upsilon_T)}^2 + \pe^2 f_\infty^2 \Vert \nabla Z \Vert_{L^2(\Upsilon_T)}^2 + \Vert G(w) \Vert_{L^2(\Upsilon_T)}^2 \Big) 
    \end{aligned}
\end{equation*}
and thus, using \eqref{eq:to get first fp est}, \eqref{eq:second bit for fp}, \eqref{eq:third bit for fp}, 
\begin{equation*}
    \begin{aligned}
        \Vert \partial_t z \Vert_{L^2(\Upsilon_T)}^2 \leq &  C_{\infty,D_e}(1+\pe^2 D_e^{-1}) C_0^3 \Big( \Vert z_0 \Vert_{H^1(\Upsilon)}^2 + \Vert G(w) \Vert_{L^2(\Upsilon_T)}^2 \Big). 
    \end{aligned}
\end{equation*}
Thus, there holds 
\begin{equation*}
    \Vert z \Vert_\Lambda \leq C_{\infty,D_e}(1+\pe^2 D_e^{-1})^{1/2} C_0^{3/2} \Big( \Vert z_0 \Vert_{H^1(\Upsilon)} + \Vert G(w) \Vert_{L^2(\Upsilon_T)} \Big). 
\end{equation*}

Finally, using the linearity of the equation to introduce difference quotients with respect to the space-angle variable $\bxi$, we reproduce the strategy required to obtain \eqref{eq:first est stat states fixed point} and get 
\begin{equation*}
    \Vert \nabla_{\bxi} z \Vert_\Lambda \leq C_{\infty,D_e} C_0^{3/2} \Big( \Vert \nabla_{\bxi} z_0 \Vert_{H^1(\Upsilon)} + \Vert \nabla_{\bxi} G(w) \Vert_{L^2(\Upsilon_T)} \Big). 
\end{equation*}
It follows that $\Vert z \Vert_{\Xi}$ is controlled by $\Vert z_0\Vert_{H^2(\Upsilon)} + \Vert G(w) \Vert_{L^2(0,T;H^1(\Upsilon))}$ as per the previous inequality, and the final estimate \eqref{eq:main fixed pt est for S} follows. 
\end{proof}

\begin{proof}[Proof of Lemma \ref{eq:parabolic estimates for fixed pt ii}]
    
    We focus on the contractivity of $G$. Direct computation yields 
\begin{equation}\label{eq:L2 exp G contractive}
   \begin{aligned}
       \Vert G(w_1)- G(w_2) \Vert_{L^2(\Upsilon_T)} = D_e \Big(& \Vert (w_1- w_2) \Delta W_1 \Vert_{L^2(\Upsilon_T)} + \Vert w_2 \Delta (W_1-W_2) \Vert_{L^2(\Upsilon_T)} \\ 
       &+ \Vert W_1 \Delta (w_1 - w_2) \Vert_{L^2(\Upsilon_T)} + \Vert (W_1-W_2) \Delta w_2 \Vert_{L^2(\Upsilon_T)} \\ 
       +\pe \Big(&\Vert (w_1-w_2) \nabla W_1 \Vert_{L^2(\Upsilon_T)} + \Vert w_2 \nabla (W_1-W_2) \Vert_{L^2(\Upsilon_T)} \\ 
       &+ \Vert W_1 \nabla (w_1-w_2) \Vert_{L^2(\Upsilon_T)} + \Vert (W_1-W_2) \nabla w_2 \Vert_{L^2(\Upsilon_T)} \Big). 
   \end{aligned}
\end{equation}
Using Morrey's embedding $H^2(\Upsilon) \hookrightarrow L^\infty(\Upsilon)$ for the bounded open domain $\Upsilon \subset \mathbb{R}^3$, there exists a positive constant $C_M$ depending only on $\Upsilon$ such that 
\begin{equation}\label{eq:apply morrey}
    \Vert w_i(t,\cdot) \Vert_{L^\infty(\Upsilon)} \leq C_M \Vert w_i(t,\cdot) \Vert_{H^2(\Upsilon)}, 
\end{equation}
whence 
\begin{equation*}
    \begin{aligned}
        \Vert w_i \Vert_{L^\infty(\Upsilon_T)} \leq C_M \Vert w_i \Vert_\Xi. 
    \end{aligned}
\end{equation*}
It follows that, for instance, 
\begin{equation*}
    \begin{aligned}
        \Vert (w_i - w_j) \Delta W_i \Vert_{L^2(\Upsilon_T)} \leq 2\pi C_M \Vert w_i - w_j \Vert_\Xi \Vert w_i \Vert_{L^2(0,T;H^2(\Upsilon))} \leq 2\pi C_M \Vert w_i - w_j \Vert_\Xi \Vert w_i \Vert_{\Xi}, 
    \end{aligned}
\end{equation*}
and thus, arguing similarly for the other terms, it follows from \eqref{eq:L2 exp G contractive} that 
\begin{equation*}
    \Vert G(w_1)- G(w_2) \Vert_{L^2(\Upsilon_T)} \leq 4 \pi C_M D_e (1+ \pe D_e^{-1} ) \Vert w_1 - w_2 \Vert_\Xi \big( \Vert w_1 \Vert_\Xi + \Vert w_2 \Vert_{\Xi} \big). 
\end{equation*}
The same strategy yields an analogous estimate for $ \Vert G(w_1)- G(w_2) \Vert_{L^\infty(0,T;L^2(\Upsilon))}$.

Observe that 
\begin{equation*}
   \begin{aligned}
       \nabla_{\bxi} G(w) = &D_e \big( w \nabla_{\bxi} \Delta W + \nabla_{\bxi} w \Delta W - \nabla_{\bxi} W \Delta w - W \nabla_{\bxi} \Delta w  \big) \\ 
       &+ \pe \big( w \nabla_{\bxi} \Delta W + \nabla_{\bxi} w \Delta W + \nabla_{\bxi} W \Delta w + W \nabla_{\bxi} \Delta w  \big) \cdot \e (\theta) \\ 
       &+ \pe \big( w \Delta W + W \Delta w \big) \cdot \e'(\theta), 
   \end{aligned} 
\end{equation*}
and, for $w \in \Xi$, arguing as per \eqref{eq:apply morrey}, Morrey's embedding implies $\nabla w \in L^2(0,T;L^\infty(\Upsilon))$; furthermore this latter norm is controlled by the norm of $\Xi$. We also have $\Delta w \in L^\infty(0,T;L^2(\Upsilon))$ and $\nabla \Delta w \in L^2(\Upsilon_T)$, and analogously for the quantities involving $W$; furthermore, all of these norms are all controlled by $\Vert w \Vert_\Xi$. Thus, by arguing as per \eqref{eq:L2 exp G contractive} we obtain 
\begin{equation*}
    \Vert \nabla_{\bxi} G(w_1) - \nabla_{\bxi} G(w_2) \Vert_{L^2(\Upsilon_T)} \leq 8 \pi C_M D_e (1+\pe D_e^{-1} ) \Vert w_1 - w_2 \Vert_\Xi \big( \Vert w_1 \Vert_\Xi + \Vert w_2 \Vert_{\Xi} \big), 
\end{equation*}
and \eqref{eq:G contractive} follows. Finally, since $G(0)=0$, the previous estimate yields \eqref{eq:self mapping for G}. 
\end{proof}

\appendix

\section{Interpolation Lemma}\label{sec:appendix interpolation}

In the present appendix, we derive a variant of a classical interpolation estimate for spaces of functions depending on time and space; our function spaces include dependence on time, space, and angle. 

\begin{lemma}\label{lem:our dibenedetto}
Let $m,p \geq 1$ and define the Banach spaces 
\begin{equation*}
    \begin{aligned}
        &V^{m,p} := L^\infty(\Omega_T;L^m(0,2\pi))) \cap L^p(\Omega_T;W^{1,p}(0,2\pi)), \\ 
        & V^{m,p}_0 := L^\infty(\Omega_T;L^m(0,2\pi)) \cap L^p(\Omega_T;W^{1,p}_0(0,2\pi)), 
    \end{aligned}
\end{equation*}
both equipped with the norm 
\begin{equation*}
    \Vert v \Vert_{V^{m,p}} := \esssup_{(t,x) \in \Omega_T} \Vert v(t,x,\cdot) \Vert_{L^m(0,2\pi)} + \Vert \partial_\theta v \Vert_{L^p(\Upsilon_T)}. 
\end{equation*}
There exists a positive constant $C$ depending only on $p,m,\Upsilon,T$ such that, for all $v \in V^{m,p}$, there holds 
\begin{equation*}
    \Vert v \Vert_{L^q(\Upsilon_T)} \leq C \Vert v \Vert_{V^{m,p}}, 
\end{equation*}
where $q = p (m+1)$. 
\end{lemma}
\begin{proof}

    \noindent 1. \textit{Inequality for $V_0^{m,p}$}: Assume for the time being that $v \in V^{m,p}_0$. Observe that $p \geq 1 > \frac{m}{m+1}$, \textit{i.e.} $p-m + mp > 0$. The Gagliardo--Nirenberg inequality \cite[\S 1, Theorem 2.1, (2.2-i)]{DiBenedetto} yields, for a.e.~$(t,x) \in \Omega_T$, 
    \begin{equation*}
        \Vert v(t,x,\cdot) \Vert_{L^q(0,2\pi)} \leq C \Vert \partial_\theta v(t,x,\cdot) \Vert_{L^p(0,2\pi)}^{\frac{p}{q}} \Vert v (t,x,\cdot) \Vert^{1-\frac{p}{q}}_{L^m(0,2\pi)}, 
    \end{equation*}
where $q=p(m+1)$. In turn, raising the previous inequality to the power $q$ and integrating over $\Omega_T$ yields 
\begin{equation*}
    \begin{aligned}
        \Vert v \Vert_{L^q(\Upsilon_T)}^q \leq C \bigg(\int_{\Omega_T} \Vert \partial_\theta v(t,x,\cdot) \Vert^{p}_{L^p(0,2\pi)}  \d x \d t \bigg) \esssup_{(t,x)\in \Omega_T}\Vert v (t,x,\cdot) \Vert^{q-p}_{L^m(0,2\pi)}, 
    \end{aligned}
\end{equation*}
and thus 
\begin{equation*}
    \begin{aligned}
        \Vert v \Vert_{L^q(\Upsilon_T)} \leq C \Vert \partial_\theta v \Vert^{\frac{p}{q}}_{L^p(\Upsilon_T)} \Vert v (t,x,\cdot) \Vert^{1-\frac{p}{q}}_{L^\infty(\Omega_T;L^m(0,2\pi))}. 
    \end{aligned}
\end{equation*}
Applying the Young inequality to the previous right-hand side with exponent $q/p>1$ and its H\"older conjugate yields 
\begin{equation*}
    \begin{aligned}
        \Vert v \Vert_{L^q(\Upsilon_T)} \leq C \Vert v \Vert_{V^{m,p}}. 
    \end{aligned}
\end{equation*}
A classical argument shows that the bound above still holds for functions $v \in V^{m,p}$ such that $\int_0^{2\pi} v \d \theta = 0$ a.e.~in $\Omega_T$; \textit{cf.}~\cite[\S1 Remark 2.1]{DiBenedetto}. 

\noindent 2. \textit{Relaxation to $V^{m,p}$}: Let $v \in V^{m,p}$ and define, for a.e.~$(t,x,\theta)$, its mean-zero counterpart 
\begin{equation*}
    w(t,x,\theta) := v(t,x,\theta) - \frac{1}{2\pi}\int_0^{2\pi} v(t,x,\theta) \d \theta. 
\end{equation*}
In view of the remark at the end of Step 1, we deduce that there holds 
\begin{equation*}
    \Vert w \Vert_{L^q(\Upsilon_T)} \leq C \Vert w \Vert_{V^{m,p}}. 
\end{equation*}
Notice that $\partial_\theta w = \partial_\theta v$ and $\Vert w(t,x,\cdot) \Vert_{L^m(0,2\pi)} \leq 2 \Vert v(t,x,\cdot) \Vert_{L^m(0,2\pi)}$ a.e.~in $\Omega_T$, and that Jensen's inequality yields 
\begin{equation*}
    \begin{aligned}
        \Big\Vert \int_0^{2\pi} v(\cdot,\cdot,\theta) \d \theta \Big\Vert_{L^q(\Upsilon_T)} \leq & C\bigg( \int_0^{2\pi} \int_{\Omega_T} \bigg(  \int_0^{2\pi} |v(t,x,\theta')|^m \d \theta' \bigg)^{\frac{q}{m}} \d x \d t \d \theta \bigg)^{\frac{1}{q}} \\ 
        \leq & C \esssup_{(t,x)\in \Omega_T} \Vert v(t,x,\cdot) \Vert_{L^m(0,2\pi)}. 
    \end{aligned}
\end{equation*}
It therefore follows from the Minkowski inequality that 
\begin{equation*}
    \begin{aligned}
       \Vert v \Vert_{L^q(\Upsilon_T)} \leq & \Vert w \Vert_{L^q(\Upsilon_T)} +  \frac{1}{2\pi} \Big\Vert \int_0^{2\pi} v \d \theta \Big\Vert_{L^q(\Upsilon_T)} \\ 
        \leq & C \bigg( \Vert \underbrace{\partial_\theta w }_{=\partial_\theta v}\Vert_{L^p(\Upsilon_T)}  + \esssup_{(t,x)\in \Omega_T} \Vert v(t,x,\cdot) \Vert_{L^m(0,2\pi)} \bigg), 
    \end{aligned}
\end{equation*}
    and the result follows. 
\end{proof}

\section{Regularised Initial Data}\label{sec:appendix initial data}.

\begin{proof}[Proof of Lemma \ref{lem:regularised initial data}]
The strong convergences in the statement of the lemma follow directly from recalling the mollification technique of \cite[Appendix C]{bbes}. It remains to verify the estimates; both those depending on $\varepsilon$ and those independent of $\varepsilon$.

\noindent 1. \textit{Bounds depending on $\varepsilon$ }: using the non-negativity of $\rho_0$, there holds 
\begin{equation*}
   0< \frac{\varepsilon}{2} \leq \rho^\varepsilon_0(x) \leq \frac{\varepsilon}{2} + (1-\varepsilon) \int_{\mathbb{R}^2} \alpha_\varepsilon \d x = 1- \frac{\varepsilon}{2} < 1, 
\end{equation*}
whence, for all $\varepsilon \in (0,1)$ and a.e.~$x \in \Omega$, 
\begin{equation}\label{eq:bounds on 1-rho initial eps}
1 - \frac{\varepsilon}{2} \geq 1-\rho^\varepsilon_0(x) \geq \frac{\varepsilon}{2}. 
\end{equation}
Similarly, we note the bound 
\begin{equation}\label{eq:bound on f initial eps}
    f^0_\varepsilon \geq \frac{\varepsilon}{4\pi}, 
\end{equation}
which holds due to the non-negativity of $f_0$.

\noindent 2. \textit{Uniform $L^2$-bound on $u^\varepsilon_0$ }: we estimate the quantity 
\begin{equation*}
   \begin{aligned}
       \varepsilon \int_\Upsilon |u^\varepsilon_0|^2 \d \bxi = \varepsilon \int_{\{0<f^\varepsilon_0/(1-\rho^\varepsilon_0)<1\}} |u^\varepsilon_0|^2 \d \bxi + \varepsilon \int_{\{f^\varepsilon_0/(1-\rho^\varepsilon_0) \geq 1\}} |u^\varepsilon_0|^2 \d \bxi =: I^\varepsilon_1 + I^\varepsilon_2. 
   \end{aligned} 
\end{equation*}
Using that there exists $C>0$ such that $|\log s|\mathds{1}_{\{0<s<1\}} \leq C |s|^{-1/2}$, the monotonicity of $s \mapsto s^{1/2}$ for $s>0$, and the estimates \eqref{eq:bounds on 1-rho initial eps}-\eqref{eq:bound on f initial eps}, the first integral is controlled by 
\begin{equation*}
    \begin{aligned}
        I^\varepsilon_1 \leq \varepsilon C\int_{\Upsilon} (1-\rho^\varepsilon_0)|{f^\varepsilon_0}|^{-1} \d \bxi \leq \varepsilon C\int_\Upsilon \varepsilon^{-1} \d \bxi \leq C, 
    \end{aligned}
\end{equation*}
and the right-hand side is uniformly bounded independently of $\varepsilon$. Similarly, using that $|\log s| \mathds{1}_{\{s \geq 1\}} \leq |s|^{\frac{1}{2}}$, there holds 
\begin{equation*}
    I^\varepsilon_2 \leq C \varepsilon \int_\Upsilon \varepsilon^{-1} f^\varepsilon_0 \d \bxi\leq C \int_\Omega \rho^\varepsilon_0 \d x \leq C, 
\end{equation*}
and the right-hand side is again uniformly bounded.

\noindent 3. \textit{Uniform bound on entropy functional}: the aforementioned strong convergences $f^\varepsilon_0 \to f_0$ and $\rho^\varepsilon_0 \to \rho_0$ imply that, up to a subsequence which we do not relabel, there holds $f^\varepsilon_0 \to f_0$ a.e.~in $\Upsilon$ and $\rho^\varepsilon_0 \to \rho_0$ a.e.~in $\Omega$. 

Since $s \mapsto s \log s$ is locally bounded, the bound $0 < \rho^\varepsilon_0 < 1$ implies the uniform boundedness of $(1-\rho^\varepsilon_0)\log(1-\rho^\varepsilon_0)$, whence the Dominated Convergence Theorem implies 
\begin{equation*}
   \lim_{\varepsilon\to 0} \int_\Omega (1-\rho^\varepsilon_0)\log(1-\rho^\varepsilon_0) \d x = \int_\Omega (1-\rho_0)\log(1-\rho_0) \d x \leq C, 
\end{equation*}
for $C$ depending only on the maximum value of $s \log s$ over the interval $[0,1]$ and $\Omega$. 

Similarly, we decompose 
\begin{equation*}
    \int_\Upsilon f^\varepsilon_0 \log f^\varepsilon_0 \d \bxi = \int_{ \{0 < f^\varepsilon_0 < 1\}} f^\varepsilon_0 \log f^\varepsilon_0 \d \bxi + \int_{ \{f^\varepsilon_0 \geq 1\}} f^\varepsilon_0 \log f^\varepsilon_0 \d \bxi, 
\end{equation*}
and the local boundedness of $s\log s$ and the Dominated Convergence Theorem that 
\begin{equation*}
    \lim_{\varepsilon \to 0}\int_{ \{0 < f^\varepsilon_0 < 1\}} f^\varepsilon_0 \log f^\varepsilon_0 \d \bxi = \int_{ \{0 < f_0 < 1\}} f_0 \log f_0 \d \bxi \leq C, 
\end{equation*}
where the right-hand side of the above is bounded by virtue of the function $s \log s$ being bounded over the interval $s \in (0,1)$. Meanwhile, the assumption $f_0 \in L^p(\Upsilon)$ for $p>1$ and the estimate $|\log s| \mathds{1}_{\{s \geq 1 \}} \leq C_p |s|^{p-1}$ imply 
\begin{equation*}
    \int_{ \{f^\varepsilon_0 > 1\}} f^\varepsilon_0 \log f^\varepsilon_0 \d \bxi \leq C_p\int_{ \{f^\varepsilon_0 > 1\}} |f^\varepsilon_0|^p \d \bxi \leq C_p(1+\Vert f_0 \Vert_{L^p(\Upsilon)}^p), 
\end{equation*}
and the right-hand side is bounded independently of $\varepsilon$. 
\end{proof}

The next proof consists in a quantitative diagonal approach; the argument is classical, but it is nevertheless useful as it helps identify exactly which diagonal subsequence we select. We recall that for $\{\sigma'(n)\}_n$ to be a subsequence of $\{\sigma(n)\}_{n\in\mathbb{N}}$, it is sufficient that there holds: $\{\sigma'(n)\}_n \subset \{\sigma(n)\}_n$, $\sigma'(n) \geq \sigma(n)$, and $\sigma'(n+1)>\sigma'(n)$ for all $n$. 

\begin{proof}[Proof of Lemma \ref{lem:convergence of initial data}]

For the purposes of the argument that follows, it is crucial that we replace the continuum $\{\varepsilon\}_{\varepsilon \in (0,1)}$ with a countable set. For this reason, we select $1>\varepsilon_1 > \varepsilon_2 > \dots > 0$ and define 
\begin{equation*}
    f^m_n(0) := f^{\varepsilon_m}_n(0), \qquad f^m_0 := f^{\varepsilon_m}_0, 
\end{equation*}
so that the limit $m\to\infty$ morally corresponds to $\varepsilon \to 0$. 

Recall from equation \eqref{eq:strong conv in Lq for initial f eps n} of \S \ref{sec:existence galerkin} that for each fixed $m$, we have 
\begin{equation*}
    \Vert f^m_n(0) - f^m_0 \Vert_{L^p(\Upsilon)} \to 0 \quad \text{as } n \to \infty. 
\end{equation*}
In turn, for $m=1$, there exists a subsequence $\{\sigma_1(n)\}_{n}$ of $\mathbb{N}$ such that, for all $n$, 
\begin{equation}\label{eq:first subsequence}
    \Vert f^1_{\sigma_1(n)}(0) - f^1_0 \Vert_{L^p(\Upsilon)} \leq \frac{1}{n}. 
\end{equation}
Note in particular that the above implies, for all $\ell \geq n$, 
\begin{equation}\label{eq:skipping steps}
    \Vert f^1_{\sigma_1(\ell)}(0) - f^1_0 \Vert_{L^p(\Upsilon)} \leq \frac{1}{n}. 
\end{equation}

Then, for $m=2$, since we also know that $\Vert f^2_{n}(0) - f^2_0 \Vert_{L^p(\Upsilon)} \to 0$ as $n\to\infty$, we may construct a subsequence $\{\sigma_2(n)\}_{n}$ of $\{\sigma_1(n)\}_n$ such that, for all $n$, 
\begin{equation*}
    \Vert f^2_{\sigma_2(n)}(0) - f^2_0 \Vert_{L^p(\Upsilon)} \leq \frac{1}{n}. 
\end{equation*}
Furthermore, since $\{\sigma_2(n)\}_n$ is a subsequence of $\{\sigma_1(n)\}_n$, the bound \eqref{eq:skipping steps} implies, for all $n$, 
\begin{equation*}
    \Vert f^1_{\sigma_2(n)}(0) - f^1_0 \Vert_{L^p(\Upsilon)} \leq \frac{1}{n}. 
\end{equation*}

We continue iteratively, and construct for each $m\in\mathbb{N}$ a subsequence $\{\sigma_m(n)\}_n$ of all previous $m-1$ subsequences, such that there holds, for all $j \in \{1,\dots,m\}$ and all $n$, 
\begin{equation}\label{eq:just before choosing the diagonal}
    \Vert f^j_{\sigma_m(n)}(0) - f^j_0 \Vert_{L^p(\Upsilon)} \leq \frac{1}{n}. 
\end{equation}
In turn, we define the diagonal sequence 
\begin{equation}\label{eq:diagonal subseq is defined}
    f^m := f^m_{\sigma_m(m)}, 
\end{equation}
and note that \eqref{eq:just before choosing the diagonal} implies, for all $m$, 
\begin{equation*}
    \Vert f^m (0) - f^m_0 \Vert_{L^p(\Upsilon)} \leq \frac{1}{m}. 
\end{equation*}
Using the above and the strong convergence of $f^m_0 \to f_0$ in $L^p(\Upsilon)$ of Lemma \ref{lem:regularised initial data}, the triangle inequality yields 
\begin{equation*}
    \Vert f^m(0) - f_0 \Vert_{L^p(\Upsilon)} \leq  \underbrace{\Vert f^m(0) - f_0^m \Vert_{L^p(\Upsilon)}}_{\leq 1/m} +  \underbrace{\Vert f_0^m - f_0 \Vert_{L^p(\Upsilon)}}_{\to 0} \to 0 
\end{equation*}
as $m\to\infty$, and the proof is complete.    
\end{proof}

\vspace{0.3cm}

\textbf{Acknowledgements: } 

MB acknowledges support from DESY (Hamburg, Germany), a member of the Helmholtz Association HGF, and the German Science Foundation (DFG) through CRC TR 154, subproject C06. SMS acknowledges support from Centro di Ricerca Matematica Ennio De Giorgi. 

The authors report there are no competing interests to declare.

\bibliographystyle{siam}
\bibliography{active.bib}

\end{document}